\RequirePackage{fix-cm}
\documentclass[oneside,english]{amsart}
\usepackage[T1]{fontenc}
\usepackage[latin9]{inputenc}
\usepackage{babel}
\usepackage{mathrsfs}
\usepackage{enumitem}
\usepackage{amstext}
\usepackage{amsthm}
\usepackage{amssymb}
\usepackage[all]{xy}
\usepackage[pdfusetitle,
 bookmarks=true,bookmarksnumbered=false,bookmarksopen=false,
 breaklinks=false,pdfborder={0 0 1},backref=false,colorlinks=false]
 {hyperref}

\makeatletter
\numberwithin{equation}{section}
\numberwithin{figure}{section}
\theoremstyle{plain}
\newtheorem{thm}{\protect\theoremname}
\theoremstyle{plain}
\newtheorem{prop}[thm]{\protect\propositionname}
\theoremstyle{plain}
\newtheorem{lem}[thm]{\protect\lemmaname}
\theoremstyle{remark}
\newtheorem{rem}[thm]{\protect\remarkname}
\theoremstyle{definition}
\newtheorem{defn}[thm]{\protect\definitionname}
\theoremstyle{definition}
\newtheorem{example}[thm]{\protect\examplename}
\theoremstyle{plain}
\newtheorem{cor}[thm]{\protect\corollaryname}

\newcommand{\xyR}[1]{\xydef@\xymatrixrowsep@{#1}}
\newcommand{\xyC}[1]{\xydef@\xymatrixcolsep@{#1}}

\newcommand{\Gal}{\mathrm{Gal}}

\newcommand{\Aut}{\mathrm{Aut}}

\newcommand{\Hom}{\mathrm{Hom}}
\newcommand{\Spec}{\mathrm{Spec}}
\newcommand{\Sp}{\mathrm{Sp}}

\newcommand{\AlgSp}{\mathsf{AlgSp}}

\newcommand{\Fet}{\mathsf{Fet}}

\newcommand{\Sch}{\mathsf{Sch}}
\newcommand{\Shv}{\mathsf{Shv}}
\newcommand{\PreShv}{\mathsf{PreShv}}

\newcommand{\Set}{\mathsf{Set}}

\newcommand{\loc}{\mathsf{loc}}

\newcommand{\Top}{\mathsf{Top}}

\newcommand{\res}{\mathrm{res}}
\newcommand{\Supp}{\mathrm{Supp}}

\makeatother

\providecommand{\corollaryname}{Corollary}
\providecommand{\definitionname}{Definition}
\providecommand{\examplename}{Example}
\providecommand{\lemmaname}{Lemma}
\providecommand{\propositionname}{Proposition}
\providecommand{\remarkname}{Remark}
\providecommand{\theoremname}{Theorem}

\begin{document}
\title{Morphing etale spaces}
\author{Christophe Cornut}
\address{Sorbonne Université, Université Paris Cité, CNRS, IMJ-PRG, F-75005
Paris, France.}
\begin{abstract}
We give a simple description of the category of sheaves on the small
etale site of an irreducible scheme whose local rings are geometrically
unibranch and henselian, which affords a characterization of representable
sheaves. 
\end{abstract}

\maketitle

\section{Introduction}

Let $S$ be a scheme, $\iota:Z\hookrightarrow S$ a closed immersion
with complementary open immersion $j:U\hookrightarrow S$. By Grothendieck's
gluing theorem \cite[IV, 9.5]{EGA4.1}, the category $\Shv(S_{et})$
of sheaves on the small etale site of $S$ is equivalent to the category
of triples $(B_{Z},B_{U},\ell_{B})$ where $B_{Z}\in\Shv(Z_{et})$,
$B_{U}\in\Shv(U_{et})$, and $\ell_{B}:B_{Z}\rightarrow\iota^{\ast}j_{\ast}B_{U}$
is a morphism in $\Shv(Z_{et})$. The equivalence takes $B\in\Shv(S_{et})$
to $B_{Z}=\iota^{\ast}B$, $B_{U}=j^{\ast}B$, and $\ell_{B}=\iota^{\ast}(u_{B})$
where $u_{B}:B\rightarrow j_{\ast}j^{\ast}B$ is the unit of the adjunction
\[
j^{\ast}:\Shv(S_{et})\rightarrow\Shv(U_{et}):j_{\ast}.
\]
This gets particularly simple when $Z$ and $U$ are punctual schemes,
corresponding to the closed and generic points $s$ and $\eta$ of
a $1$-dimensional irreducible local scheme $S$: by the topological
invariance of etale sheaves, the first two components of our triples
may then be viewed as etale sheaves on the corresponding residue fields,
i.e.~as sets equipped with a smooth action of the corresponding absolute
Galois groups. The description of the connecting morphism $\ell_{B}$
is however somewhat trickier. 

When $S$ is the spectrum of a henselian discrete valuation ring $\mathcal{O}$,
we arrive at the following picture. Let $K$ be the fraction field
of $\mathcal{O}$, $K^{sep}$ a separable closure of $K$, $K^{nr}$
the maximal unramified extension of $K$ in $K^{sep}$, $\mathcal{O}^{nr}$
the integral closure of $\mathcal{O}$ in $K^{nr}$, $G=\Gal(K^{sep}/K)$
the Galois group and $I=\Gal(K^{sep}/K^{nr})$ the inertia subgroup,
so that $G/I=\Gal(K^{nr}/K)\simeq\Gal(k^{sep}/k)$ where $k^{sep}$
is the residue field of $\mathcal{O}^{nr}$, a separable closure of
the residue field $k$ of $\mathcal{O}$. Then $\Shv(S_{et})$ is
equivalent to the category of morphisms of smooth $G$-sets $\ell_{B}:B_{\overline{s}}\rightarrow B_{\overline{\eta}}$
where $I$ acts trivially on $B_{\overline{s}}$. A sheaf $B$ is
mapped to the localization morphism between its stalks at the geometric
points $\overline{s}$ and $\overline{\eta}$ of $S=\Spec(\mathcal{O})$
which are respectively determined by 
\[
\xyR{1pt}\xymatrix{ & \mathcal{O}^{nr}\ar@{->>}[rd] &  &  &  & \mathcal{O}^{nr}\ar@{^{(}->}[rd]\\
\mathcal{O}\ar@{^{(}->}[ru]\ar@{->>}[rd] &  & k^{sep} & \text{and} & \mathcal{O}\ar@{^{(}->}[ru]\ar@{_{(}->}[rd] &  & K^{nr}\ar@{^{(}->}[r] & K^{sep}.\\
 & k\ar@{_{(}->}[ru] &  &  &  & K\ar@{_{(}->}[ru]
}
\]
When $B=\Hom_{S}(-,X)$ for some $X\in S_{et}$, we obtain the morphism
\[
\xyR{2pc}\xymatrix{X(k^{sep})\ar[r]^{\ell_{B}} & X(K^{sep})\\
X(\mathcal{O}^{nr})\ar[u]^{\simeq}\ar[r] & X(K^{nr})\ar@{_{(}->}[u]
}
\]
When $X$ belongs to the strictly full subcategory $S_{set}$ of separated
etale $S$-schemes, the morphism $\ell_{B}$ is injective by the valuative
criterion of separatedness. In general, covering $X$ by affines,
we find that $\ell_{B}$ is injective on $G$-orbits, i.e.~
\[
\forall x\in X(k^{sep}),\quad\forall g\in G:\qquad\ell_{B}(gx)=\ell_{B}(x)\iff gx=x.
\]
It turns out that the converse implications hold. Namely, a sheaf
$B\in\Shv(S_{et})$ is representable by some $X\in S_{set}$ (resp.
by some $X\in S_{et}$) if and only if $\ell_{B}:B_{\overline{s}}\rightarrow B_{\overline{\eta}}$
is injective (resp.~injective on orbits). In particular, there are
adjunctions 
\[
(-)_{set}:\Shv(S_{et})\longleftrightarrow S_{set}:\mathrm{yon}
\]
\[
(-)_{et}:\Shv(S_{et})\longleftrightarrow S_{et}:\mathrm{yon}
\]
where $\mathrm{yon}$ is the Yoneda embedding, and the left adjoints
$(-)_{set}$ and $(-)_{et}$ correspond to the functors which map
a $G$-morphism $\ell_{B}:B_{\overline{s}}\rightarrow B_{\overline{\eta}}$
to respectively
\[
\xyC{3pc}\xymatrix{(\ell_{B})_{set}:\mathrm{Im}(B_{\overline{s}}\rightarrow B_{\overline{\eta}})\ar@{^{(}->}[r]\sp(0.71){\mathrm{inc}} & B_{\overline{\eta}}}
\]
\[
\xymatrix{(\ell_{B})_{et}:\mathrm{Im}(B_{\overline{s}}\rightarrow B_{\overline{\eta}}\times G\backslash B_{\overline{s}})\ar[r]\sp(0.75){p_{1}} & B_{\overline{\eta}}}
\]

The modest goal of this paper is to explain and generalize this to
the case where $S$ is irreducible with geometrically unibranch henselian
local rings, e.g.~$S=\Spec(\mathcal{O})$ where $\mathcal{O}$ is
an arbitrary henselian valuation ring. Sections $2$ to $6$ investigate
the relations between various notions of etale objects over $S$:
etale algebraic spaces, etale sheaves, fet-sheaves (which are sheaves
on the subsite $S_{fet}$ of $S_{et}$ whose objects are finite etale
over opens of $S$), Zariski sheaves of finite etale sheaves, Zariski
sheaves of $\pi$-sets or $G$-sets. Under good assumptions on $S$,
we arrive at a fairly concrete strictly full subcategory $\Shv_{G}^{\star}(S_{Zar})$
of the category of Zariski sheaves of $G$-sets, which is equivalent
to all categories previously considered. We explore its features in
section $7$, and collect our findings in section $8$. Section~$9$
spells them out when $S=\Spec(\mathcal{O})$ for a henselian valuation
ring $\mathcal{O}$. 

\section{From etale algebraic spaces to etale sheaves}

We fix a big fppf site $(\Sch/S)_{fppf}$ as defined in \cite[\href{https://stacks.math.columbia.edu/tag/021L}{Tag 021L}]{SP}
and let $S_{et}$ be the corresponding small etale site, whose underlying
category is the strictly full subcategory of $X$'s in $(\Sch/S)_{fppf}$
which are etale over $S$, equipped with the induced topology. The
embedding $S_{et}\hookrightarrow(\Sch/S)_{fppf}$ induces a morphism
of sites
\[
\theta:(\Sch/S)_{fppf}\rightarrow S_{et},
\]
whence an adjunction between the corresponding pull-back and push-out
functors on sheaves, both of which $2$-commute with the Yoneda embeddings:
\[
\xymatrix{ & S_{et}\ar@/_{1pc}/[dl]_{\mathrm{yon}}\ar@/^{1pc}/[dr]^{\mathrm{yon}}\\
\qquad\quad\Shv(S_{et})\ar@<2pt>[rr]^{\theta^{\ast}} &  & \Shv((\Sch/S)_{fppf})\ar@<2pt>[ll]^{\theta_{\ast}}
}
\]
Let $\AlgSp(S)$ be the strictly full subcategory of algebraic spaces
in $\Shv((\Sch/S)_{fppf})$, and $\AlgSp_{et}(S)$ the strictly full
subcategory of algebraic spaces etale over $S$. 
\begin{prop}
\label{prop:FromAlgSp2EtShv}The above adjunction restricts to mutually
inverse equivalences
\[
\theta^{\ast}:\Shv(S_{et})\longleftrightarrow\AlgSp_{et}(S):\theta_{\ast}
\]
and induces an adjunction 
\[
\mathrm{inc}:\AlgSp_{et}(S)\longleftrightarrow\Shv((\Sch/S)_{fppf}):\theta^{\ast}\theta_{\ast}.
\]
\end{prop}

\begin{proof}
It is sufficient to establish that for $A\in\AlgSp_{et}(S)$ and $B\in\Shv(S_{et})$, 
\begin{enumerate}
\item The counit $\theta^{\ast}\theta_{\ast}A\rightarrow A$ is an isomorphism.
\item $\theta^{\ast}B$ belongs to $\AlgSp_{et}(S)$ and the unit $B\rightarrow\theta_{\ast}\theta^{\ast}B$
is an isomorphism.
\end{enumerate}
For $X\in S_{et}$, we denote by $H_{X}\in\Shv((\Sch/S)_{fppf})$
and $h_{X}\in\Shv(S_{et})$ the images of $X$ under the Yoneda embeddings.
Thus $\theta_{\ast}H_{X}=h_{X}$ and $\theta^{\ast}h_{X}\simeq H_{X}$. 

Since $A$ is an etale algebraic space over $S$, there is a $U\in S_{et}$
and a section $a\in A(U)$ such that the corresponding morphism $a:H_{U}\rightarrow A$
is etale surjective, i.e.~relatively representable by etale surjective
morphisms of schemes. In particular, $a:H_{U}\rightarrow A$ is an
epimorphism in $\Shv((\Sch/S)_{et})$. So it is also an epimorphism
in $\Shv((\Sch/S)_{fppf})$, and $\theta_{\ast}a$, which is the morphism
$a:h_{U}\rightarrow\theta_{\ast}A$ corresponding to $a$ in $\theta_{\ast}A(U)=A(U)$,
is an epimorphism in $\Shv(S_{et})$. By general properties of topoi,
these epimorphisms induce isomorphisms $H_{U}/\Sigma\simeq A$ and
$h_{U}/\sigma\simeq\theta_{\ast}A$, where $\Sigma=H_{U}\times_{A}H_{U}$
and $\sigma=h_{U}\times_{\theta_{\ast}A}h_{U}$ are the induced equivalence
relations on $H_{U}$ and $h_{U}$. Since $\theta_{\ast}$ is a right
adjoint, it commutes with all limits, so $\sigma=\theta_{\ast}\Sigma$.
Since $a:H_{U}\rightarrow A$ is etale surjective, so are both projections
$\Sigma\rightarrow H_{U}$. In particular, $\Sigma$ is representable
by a scheme $R$ which is etale over $U$, hence etale over $S$,
i.e.~$\Sigma=H_{R}$ and $\sigma=h_{R}$ with $R\in S_{et}$. Since
$\theta^{\ast}$ is exact, $\theta^{\ast}\theta_{\ast}A$ is the quotient
of $\theta^{\ast}h_{U}\simeq H_{U}$ by $\theta^{\ast}h_{R}\simeq H_{R}$,
i.e.~$\theta^{\ast}\theta_{\ast}A\simeq H_{U}/H_{R}\simeq A$. One
checks that the isomorphism $\theta^{\ast}\theta_{\ast}A\rightarrow A$
thus constructed is the counit of our adjunction, and this proves
$(1)$. 

Fix a set $\mathcal{S}$ of generators of $S_{et}$, for instance
the set of all standard etale affine schemes over affine open subschemes
of $S$. Let $\mathcal{B}$ be the set of pairs $(X,x)$ with $X\in\mathcal{S}$
and $x\in B(X)$. Set $U=\coprod_{(X,x)\in\mathcal{B}}X$, so that
$U\in S_{et}$. Let $b:h_{U}\rightarrow B$ be the morphism of etale
sheaves on $S$ corresponding to the section $b\in B(U)$ whose restriction
to the $(X,x)$-component $X$ of $U$ equals $x\in B(X)$. Then $b:h_{U}\rightarrow B$
is an epimorphism in $\Shv(S_{et})$. By general properties of topoi,
it induces an isomorphism $h_{U}/\sigma\simeq B$ where $\sigma=h_{U}\times_{B}h_{U}$
is the equivalence relation on $h_{U}$ induced by $b$. By lemma~\ref{lem:MonoInShvetAreRepByOpenIm}
below applied to the diagonal of $B$, the etale sheaf $\sigma$ is
representable by an open subscheme $R$ of $U\times_{S}U$, so $R\in S_{et}$
is an etale equivalence relation on $U\in S_{et}$ and $\sigma=h_{R}$.
Since $B\simeq h_{U}/h_{R}$ and $\theta^{\ast}$ is exact, $\theta^{\ast}B\simeq H_{U}/H_{R}$,
which belongs to $\AlgSp_{et}(S)$. As above, $\theta_{\ast}\theta^{\ast}B\simeq h_{U}/h_{R}\simeq B$,
and the isomorphism $B\rightarrow\theta_{\ast}\theta^{\ast}B$ thus
constructed is the unit of our adjunction. This proves $(2)$. 
\end{proof}
\begin{lem}
\label{lem:MonoInShvetAreRepByOpenIm}Any monomorphism of $\Shv(S_{et})$
is representable by open immersions. 
\end{lem}

\begin{proof}
Let $B'$ be a subsheaf of $B\in\Shv(S_{et})$, fix $T\in S_{et}$
and $b\in B(T)$. Since $B'$ is a Zariski subsheaf of $B$, there
is a largest open $U$ of $T$ such that $b\vert_{U}\in B'(U)$. Let
$f:T'\rightarrow T$ be any morphism in $S_{et}$ such that $c=f^{\ast}b\in B(T')$
belongs to $B'(T')$. Let $V=f(T')$ be the image of $f$. Since $f$
is etale, $V$ is open in $T$, $f:T'\twoheadrightarrow V$ is an
etale covering, and $T''=T'\times_{V}T'$ equals $T'\times_{T}T'$.
Let $p_{i}:T''\rightarrow T'$ be the projections. Since $p_{1}^{\ast}f^{\ast}b=p_{2}^{\ast}f^{\ast}b$
in $B(T'')$, $p_{1}^{\ast}c=p_{2}^{\ast}c$ in $B'(T'')$; since
$B'$ is a sheaf, $c=f^{\ast}b'$ for some $b'\in B'(V)$; since $B$
is a sheaf and $f^{\ast}b'=c=f^{\ast}(b\vert_{V})$ in $B(T')$, $b'=b\vert_{V}$
in $B(V)$. So $b\vert_{V}=b'$ belongs to $B'(V)$ and $V\subset U$.
It follows that $U\hookrightarrow T$ represents $B^{\prime}\times_{B}T\hookrightarrow T$,
which proves the lemma.
\end{proof}
For later use, we also record here the following consequence.
\begin{prop}
\label{prop:RepCrit4et}If $B=\cup B_{i}$ in $\Shv(S_{et})$ with
$B_{i}$ representable by $(X_{i},b_{i})$, $X_{i}\in S_{et}$, $b_{i}\in B_{i}(X_{i})$,
then $B$ is representable by $(X,b)$ for some $X\in S_{et}$, $b\in B(X)$,
$B_{i}\hookrightarrow B$ is representable by an open immersion $X_{i}\hookrightarrow X$,
with $X=\cup X_{i}$ and $b\vert_{X_{i}}=b_{i}$.
\end{prop}

\begin{proof}
This follows from the previous lemma and a variant of \cite[\href{https://stacks.math.columbia.edu/tag/01JJ}{Tag 01JJ}]{SP}.
\end{proof}
We denote by $\alpha$ the restriction of $\theta_{\ast}$ to $\AlgSp_{et}(S)$,
an equivalence of categories 
\[
\alpha:\AlgSp_{et}\left(S\right)\rightarrow\Shv\left(S_{et}\right).
\]

\begin{rem}
\label{rem:RepAlgSpIsRepEtale}Let $A\in\AlgSp_{et}(S)$ and $B=\alpha(A)\in\Shv(S_{et})$.
If $A$ is a scheme, i.e.~$A$ is representable by an $S$-scheme
$X$, then $X\rightarrow S$ is etale, i.e.~$X\in S_{et}$, and $X$
represents $B$. Conversely if $B$ is representable by $X\in S_{et}$,
then $X$, viewed as an algebraic space over $S$, is etale over $S$
and $\alpha(X)\simeq B$, so $A\simeq X$, i.e. $A$ is a scheme.
\end{rem}

\begin{prop}
\label{prop:BaseChange4alpha}For a morphism $f:S^{\prime}\rightarrow S$,
there is a $2$-commutative diagram
\[
\xymatrix{\Shv(S_{et})\ar@<2pt>[r]^{\theta^{\ast}}\ar[d]_{f^{\ast}} & \AlgSp_{et}(S)\ar@<2pt>[l]^{\alpha}\ar@{^{(}->}[r]\sp(0.42){\mathrm{inc}}\ar[d]^{f^{\ast}} & \Shv((\Sch/S)_{fppf})\ar[d]^{f^{\ast}}\\
\Shv(S_{et}^{\prime})\ar@<2pt>[r]^{\theta^{\ast}} & \AlgSp_{et}(S^{\prime})\ar@<2pt>[l]^{\alpha}\ar@{^{(}->}[r]\sp(0.42){\mathrm{inc}} & \Shv((\Sch/S')_{fppf})
}
\]
with right adjoint $2$-commutative diagram
\[
\xymatrix{\Shv(S_{et})\ar@<2pt>[r]^{\theta^{\ast}} & \AlgSp_{et}(S)\ar@<2pt>[l]^{\alpha} & \Shv((\Sch/S)_{fppf})\ar[l]\sb(0.55){\theta^{\ast}\theta_{\ast}}\\
\Shv(S_{et}^{\prime})\ar@<2pt>[r]^{\theta^{\ast}}\ar[u]^{f_{\ast}} & \AlgSp_{et}(S^{\prime})\ar@<2pt>[l]^{\alpha}\ar[u]_{f_{\ast}^{et}} & \Shv((\Sch/S')_{fppf})\ar[l]\sb(0.55){\theta^{\ast}\theta_{\ast}}\ar[u]_{f_{\ast}}
}
\]
\end{prop}

\begin{proof}
Consider the commutative diagram of morphisms of sites
\[
\xymatrix{(\Sch/S)_{fppf}\ar[r]\sp(0.6){\theta} & S_{et}\\
(\Sch/S')_{fppf}\ar[u]^{f}\ar[r]\sb(0.6){\theta} & S_{et}^{\prime}\ar[u]_{f}
}
\]
whose underlying continuous functors, going in opposite directions,
are given by 
\[
f(X)=X\times_{S}S'\qquad\text{and}\qquad\theta(X)=X.
\]
The corresponding pull-back and push-out functors respectively give
the forward and backward outer rectangles of our diagrams, and the
$2$-commutativity of all remaining possible squares follow from proposition~\ref{prop:FromAlgSp2EtShv}. 
\end{proof}
\begin{rem}
\label{rem:ZZZpushoutNOTequal}While $f^{\ast}:\AlgSp_{et}(S)\rightarrow\AlgSp_{et}(S')$
is the restriction of the eponymous functor on fppf sheaves, its right
adjoint $f_{\ast}^{et}:\AlgSp_{et}(S')\rightarrow\AlgSp_{et}(S)$
is given by 
\[
\xymatrix{\Shv(S_{et})\ar[r]^{\theta^{\ast}} & \AlgSp_{et}(S) & \Shv((\Sch/S)_{fppf})\ar[l]\sb(0.55){\theta^{\ast}\theta_{\ast}}\\
\Shv(S_{et}^{\prime})\ar[u]^{f_{\ast}} & \AlgSp_{et}(S^{\prime})\ar[l]_{\alpha}\ar@{^{(}->}[r]\sp(0.42){\mathrm{inc}}\ar[u]_{f_{\ast}^{et}} & \Shv((\Sch/S')_{fppf})\ar[u]_{f_{\ast}}
}
\]
\end{rem}

Let $\overline{s}\rightarrow S$ be a geometric point of $S$ over
$s\in S$, $k(s,\overline{s})$ the separable closure of $k(s)$ in
$k(\overline{s})$, $\mathcal{O}_{S,\overline{s}}^{sh}$ the strict
henselization of $\mathcal{O}_{S,s}$ with respect to $k(s)\hookrightarrow k(\overline{s})$.
Then $k(s,\overline{s})$ is the residue field of $\mathcal{O}_{S,\overline{s}}^{sh}$
and the action of $\Gamma(s)=\Gal(k(s,\overline{s})/k(s))$ on $k(s,\overline{s})$
lifts uniquely to a continuous action on the local ring $\mathcal{O}_{S,\overline{s}}^{sh}$.
Evaluation at $\overline{s}$ gives points of the topoi $\Shv((\Sch/S)_{fppf})$
and $\Shv(S_{et})$, with stalks
\[
\xyC{3pc}\xymatrix{\Shv(S_{et})\ar[r]\sp(0.4){\theta^{\ast}} & \Shv((\Sch/S)_{fppf})\ar[r]\sp(0.57){(-)(\overline{s})} & \Set_{\Aut(\overline{s}/s)}\ar[r]\sp(0.57){\mathrm{forget}} & \Set}
\]
where for any group $H$, $\Set_{H}$ is the category of sets with
a left action of $H$. 
\begin{prop}
\label{prop:CompStalks4Alpha}For $A$ in $\AlgSp_{et}(S)$ with image
$B=\alpha(A)$ in $\Shv(S_{et})$, there are functorial $\Aut(\overline{s}/s)$-equivariant
isomorphisms of $\Gamma(s)$-sets
\[
B_{\overline{s}}=\underrightarrow{\lim}_{(X,x)\in S_{et}(\overline{s})}B(X)\simeq A(\mathcal{O}_{S,\overline{s}}^{sh})\simeq A\left(k(s,\overline{s})\right)\simeq A(\overline{s})
\]
where $S_{et}(\overline{s})$ is the category of pairs $(X,x)$ with
$X\in S_{et}$, $x\in X(\overline{s})$. 
\end{prop}

\begin{proof}
For $(X,x)\in S_{et}(\overline{s})$, there is a canonical factorization
of $x:\overline{s}\rightarrow X$ as
\[
\overline{s}\twoheadrightarrow\Spec\left(k(s,\overline{s})\right)\hookrightarrow\Spec\left(\mathcal{O}_{S,\overline{s}}^{sh}\right)\rightarrow X.
\]
Evaluating on $A$ and taking colimits gives an $\Aut(\overline{s}/s)$-equivariant
sequence 
\[
\underrightarrow{\lim}_{(X,x)\in S_{et}(\overline{s})}B(X)\rightarrow A(\mathcal{O}_{S,\overline{s}}^{sh})\rightarrow A\left(k(s,\overline{s})\right)\rightarrow A(\overline{s}).
\]
We have to show that all maps are bijections. In the colimit, we may
restrict the indexing category to the full initial category $U_{et}^{af}(\overline{s})$
of pairs $(X,x)$ where $X$ is affine over some fixed affine neighborhood
$U$ of $s$ in $S$. Then by \cite[\href{https://stacks.math.columbia.edu/tag/04GW}{Tag 04GW}]{SP},
\[
\Spec\left(\mathcal{O}_{S,\overline{s}}^{sh}\right)=\underleftarrow{\lim}_{(X,x)\in U_{et}^{af}(\overline{s})}X\quad\text{in}\quad\Sch/S.
\]
On the other hand, $A$ is locally of finite presentation over $S$
by \cite[\href{https://stacks.math.columbia.edu/tag/0468}{Tag 0468}]{SP},
so
\[
\underrightarrow{\lim}_{(X,x)\in S_{et}(\overline{s})}B(X)\stackrel{\simeq}{\longrightarrow}A(\mathcal{O}_{S,\overline{s}}^{sh})
\]
by~\cite[\href{https://stacks.math.columbia.edu/tag/01ZC}{Tag 01ZC}]{SP}.
For $A=\Hom_{S}(-,X)$ with $X\in S_{et}$, we have isomorphisms
\[
A\left(\mathcal{O}_{S,\overline{s}}^{sh}\right)\stackrel{\simeq}{\longrightarrow}A\left(k(s,\overline{s})\right)\stackrel{\simeq}{\longrightarrow}A(\overline{s})
\]
by \cite[18.5.4.4]{EGA4.4} for the first map, and using that the
fiber $X_{s}\rightarrow s$ is a disjoint union of spectra of finite
separable extensions of $k(s)$ for the second map. For a general
$A$, choose a presentation $A\simeq(U/R)_{fppf}$ with $U,R\in S_{et}$.
It is now sufficient to establish that for any $\mathcal{O}$ in $\{\mathcal{O}_{S,\overline{s}}^{sh},k(s,\overline{s}),k(\overline{s})\}$,
the map $U(\mathcal{O})\rightarrow A(\mathcal{O})$ identifies $A(\mathcal{O})$
with the quotient of $U(\mathcal{O})$ by the equivalence relation
$R(\mathcal{O})$. Since $U\times_{A}U=R$ as presheaves on $\Sch/S$,
$U(\mathcal{O})/R(\mathcal{O})\rightarrow A(\mathcal{O})$ is injective.
Since $U\rightarrow A$ is etale surjective and $\mathcal{O}$ is
strictly henselian, it is also surjective. 
\end{proof}

\section{From etale sheaves to finite etale sheaves}
\begin{defn}
A morphism of schemes $f:X\rightarrow Y$ is \emph{fet} if it factors
as $X\rightarrow U\hookrightarrow Y$ where $f':X\rightarrow U$ is
finite etale and $U\hookrightarrow Y$ is an open immersion.
\end{defn}

Note that $f^{\prime}(X)$ is then clopen in $U$, and open in $Y$.
In particular, we may always take $U=f(X)$. Thus $f:X\rightarrow Y$
is fet if and only if $f(X)$ is open in $Y$ and $f:X\rightarrow f(X)$
is finite etale. A fet morphism is separated and etale, and an etale
morphism $f$ is fet if and only if $X\rightarrow f(X)$ is finite.
Fet morphisms are plainly stable under arbitrary base change. A morphism
$f:X\rightarrow Y$ between a fet $S$-scheme $X$ and a separated
etale $S$-scheme $Y$ is fet, and its image $f(X)$ is fet over $S$.
Indeed if $U$ is the image of $X$ in $S$, then $f$ factors as
$X\rightarrow Y_{U}\hookrightarrow Y$. Since $X$ is finite etale
over $U$ and $Y_{U}$ is separated etale over $U$, $X\rightarrow Y_{U}$
is finite etale and $f(X)$ is finite etale over $U$, so $X\rightarrow Y$
is fet and $f(X)\rightarrow S$ is fet. 
\begin{defn}
We denote by $S_{fet}$ the site whose underlying category is the
strictly full subcategory of fet $S$-schemes in $S_{et}$, equipped
with the induced topology.
\end{defn}

By~\cite[III, Corollaire 3.3]{SGA4.1} and the next lemma, coverings
in $S_{fet}$ are just coverings in $S_{et}$, i.e.~jointly surjective
families of morphisms in $S_{fet}$ with fixed target. 
\begin{lem}
The category $S_{fet}$ is stable under fiber products. 
\end{lem}

\begin{proof}
Let $X_{1}\rightarrow X_{3}$ and $X_{2}\rightarrow X_{3}$ be morphisms
in $S_{fet}$, $U_{i}$ the image of $X_{i}$ in $S$, and $V=U_{1}\cap U_{2}$.
The cartesian diagram 
\[
\xyR{2pc}\xyC{2pc}\xymatrix{X_{1}\times_{X_{3}}X_{2}\ar[r]\ar[d] & X_{2,V}\ar[r]\ar[d] & X_{2}\ar[d]\\
X_{1,V}\ar[r]\ar[d] & X_{3,V}\ar@{^{(}->}[r]\ar@{^{(}->}[d] & X_{3,U_{2}}\ar@{^{(}->}[d]\\
X_{1}\ar[r] & X_{3,U_{1}}\ar@{^{(}->}[r] & X_{3}
}
\]
shows that $X_{1}\times_{X_{3}}X_{2}$ is finite etale over $X_{3,V}=X_{3}\times_{U_{3}}V$
which is finite etale over $V$, so $X_{1}\times_{X_{3}}X_{2}$ is
finite etale over $V$ and indeed fet over $S$.
\end{proof}
We denote by $\Shv(S_{fet})$ the category of fet-sheaves, i.e. sheaves
on $S_{fet}$, and let
\[
\beta:\Shv(S_{et})\rightarrow\Shv(S_{fet})
\]
be the restriction functor. This is usually badly behaved. 
\begin{example}
Let $\mathcal{O}$ be the local ring of $\mathbb{Z}[X]/(X^{2}+1)$
at the prime $P=(X-2)$ above $p=5$. Take $S=\Spec(\mathbb{Z}_{(p)})$,
$X=\Spec(\mathcal{O})$, $X'$ the $S$-scheme obtained by gluing
two copies of $X$ along its generic fiber $Y$, $\iota:Y\hookrightarrow X$
and $a,b:X\hookrightarrow X'$ the corresponding open embeddings.
Let $\iota:\mathcal{Y}\hookrightarrow\mathcal{X}$ and $a,b:\mathcal{X}\rightarrow\mathcal{X}'$
be the induced morphisms between the corresponding representable etale
sheaves on $S$. Then $\beta(\iota)$ is an isomorphism while $\iota$
is not, so $\beta$ does not reflect isomorphisms; $\beta(a)=\beta(b)$
while $a\neq b$, so $\beta$ is not faithful; and the nontrivial
automorphism of $\beta(\mathcal{Y})=\beta(\mathcal{X})$ does not
lift to any morphism of $\mathcal{X}$, so $\beta$ is not full. Note
that $\mathbb{Z}_{(p)}$ is normal, local\ldots{} but not henselian:
this is necessary by corollary~\ref{cor:BetaEquiv} below.
\end{example}

\begin{defn}
We say that $S$ is locally henselian (resp. locally strictly henselian)
if $\mathcal{O}_{S,s}$ is henselian (resp.~strictly henselian) for
every $s\in S$. 
\end{defn}

\begin{example}
\label{exa:OfLocHensel}The spectrum of an absolutely integrally closed
ring is locally strictly henselian \cite[\href{https://stacks.math.columbia.edu/tag/0DCS}{Tag 0DCS}]{SP}.
An integral normal scheme with separably closed function field is
locally strictly henselian \cite[\href{https://stacks.math.columbia.edu/tag/09Z9}{Tag 09Z9}]{SP}.
A punctual scheme is locally henselian \cite[\href{https://stacks.math.columbia.edu/tag/06RS}{Tag 06RS}]{SP}.
A henselian valuation ring is locally henselian: this follows from
Gabber's criterion for henselian pairs \cite[\href{https://stacks.math.columbia.edu/tag/09XI}{Tag 09XI}]{SP}.
\end{example}

\begin{prop}
\label{prop:IfLocHensEtcovbyFet}Suppose that $S$ is locally henselian.
Then any object of $S_{et}$ has a Zariski covering by objects of
$S_{fet}.$
\end{prop}

\begin{proof}
Let $f:X\rightarrow S$ be an etale morphism, $x\in X$, $s=f(x)$.
We have to find an open neighborhood $U$ of $x$ in $X$ such that
$f:U\rightarrow f(U)$ is finite etale. Shrinking $S$ and $X$, we
may assume that both are affine, in which case $f$ is affine, with
finite fibers. Let $X(s)\rightarrow S(s)$ be the base change of $f$
to $S(s)=\Spec(\mathcal{O}_{S,s})$. By~\cite[2.3.2]{MoBa23}, there
is a clopen decomposition $X(s)=X(s)^{f}\coprod X(s)^{\prime}$ with
$X(s)^{f}\rightarrow S(s)$ finite (and etale) and $X(s)_{s}^{\prime}=\emptyset$
-- so $x\in X(S)^{f}.$ By~\cite[\S 8]{EGA4.3}, shrinking $S$
further around $s$, we may assume that it comes from a clopen decomposition
$X=X^{f}\coprod X'$ with $X^{f}$ finite over $S$. Then $U=X^{f}$
is the desired neighborhood of $x$ in $X$. 
\end{proof}
\begin{cor}
\label{cor:BetaEquiv}If $S$ is locally henselian, then $\beta$
is an equivalence of categories. 
\end{cor}

\begin{proof}
This now follows from the comparison lemma of \cite[III, Théorème 4.1]{SGA4.1}.
\end{proof}
\begin{cor}
\label{cor:FetShevRepbyEt}Suppose that $S$ is locally henselian.
Then $B\in\Shv(S_{et})$ is representable if and only if $\beta(B)\in\Shv(S_{fet})$
is a union of representable subpresheaves.
\end{cor}

\begin{proof}
Suppose $B=\Hom_{S_{et}}(-,X)$ for some $X\in S_{et}$. Let $X=\cup X_{i}$
be a Zariski covering of $X$ by objects $X_{i}\in S_{fet}$ and set
$B_{i}=\Hom_{S_{et}}(-,X_{i})$, a subsheaf of $B$. Then $B=\cup B_{i}$
in $\Shv(S_{et})$, so $\beta(B)=\cup\beta(B_{i})$ in $\Shv(S_{fet})$,
with $\beta(B_{i})\in\Shv(S_{fet})$ representable by $X_{i}\in S_{fet}$.
Suppose conversely that $B'=\beta(B)$ is a union of representable
subpresheaves: there is a collection of objects $X_{i}\in S_{fet}$
and sections $b_{i}\in B'(X_{i})$ inducing monomorphisms $(-)^{\ast}b_{i}:\Hom_{S_{fet}}(-,X_{i})\hookrightarrow B'$
with image $B_{i}^{\prime}\subset B'$ such that $B'=\cup B_{i}^{\prime}$
in $\Shv(S_{fet})$. Then the sections $b_{i}\in B(X_{i})$ induce
monomorphisms $(-)^{\ast}b_{i}:\Hom_{S_{et}}(-,X_{i})\hookrightarrow B$
with image $B_{i}\subset B$ such that $B=\cup B_{i}$ in $\Shv(S_{et})$.
Note that $b_{i}\in B_{i}(X_{i})$ and $(X_{i},b_{i})$ represents
$B_{i}$. Then by proposition~\ref{prop:RepCrit4et}, $B$ is representable
by a pair $(X,b)$, $X\in S_{et}$, $b\in B(X)$, with $B_{i}\subset B$
representable by an open embedding $X_{i}\hookrightarrow X$, with
$X=\cup X_{i}$ and $b\vert_{X_{i}}=b_{i}$.
\end{proof}
\begin{prop}
\label{prop:IntegralOverLocHensIsHomeo}Let $f:X\rightarrow S$ be
an integral morphism with $X$ irreducible and $S$ locally henselian.
Then $f(X)$ is closed and $f:X\rightarrow f(X)$ is an homeomorphism. 
\end{prop}

\begin{proof}
Since $f$ is universally closed \cite[\href{https://stacks.math.columbia.edu/tag/01WM}{Tag 01WM}]{SP},
we just have to show that it is injective. Fix $s\in f(S)$. Our assumptions
on $f$ are stable under base change to $\Spec(\mathcal{O}_{S,s})$,
so we may assume that $S$ is local henselian with closed point $s$.
By~\cite[\href{https://stacks.math.columbia.edu/tag/09XI}{Tag 09XI}]{SP},
any idempotent of $\Gamma(X_{s},\mathcal{O}_{X_{s}})$ lifts to an
idempotent of $\Gamma(X,\mathcal{O}_{X})$. Since $X$ is connected,
it follows that $X_{s}$ is connected. But $X_{s}$ is also totally
disconnected by~\cite[\href{https://stacks.math.columbia.edu/tag/00GS}{Tag 00GS} and \href{https://stacks.math.columbia.edu/tag/04MG}{Tag 04MG}]{SP},
so it must be a single point.
\end{proof}

\section{From fet-sheaves to Zariski Fet-sheaves}

Let $S_{Zar}$ be the usual Zariski site of $S$. For $U\in S_{Zar}$,
let $\Fet_{U}$ be the category of finite etale $U$-schemes, which
we view as a strictly full subcategory of $S_{fet}$ and equip with
the induced topology. Since $\Fet_{U}$ is stable under fiber products,
coverings in $\Fet_{U}$ still correspond to jointly surjective families
of morphisms, and for $V\subset U$, the base change functor $\Fet_{U}\rightarrow\Fet_{V}$
is continuous. Plainly $S_{fet}=\cup_{U}\Fet_{U}$. For $V=\emptyset$,
$\Fet_{\emptyset}$ is the initial object $\emptyset$ of $S_{fet}$,
$\Shv(\Fet_{\emptyset})$ is the punctual category with a single sheaf
$B_{\emptyset}$ whose direct image under $\emptyset\rightarrow U$
is the final object of $\Shv(\Fet_{U})$. 
\begin{defn}
A Zariski $\Fet$-sheaf is given by the following data:
\begin{enumerate}
\item For $U\in S_{Zar}$, a sheaf $B_{U}\in\Shv(\Fet_{U})$,
\item For $j_{V}^{U}:V\hookrightarrow U$ in $S_{Zar}$, a morphism $r_{U}^{V}:B_{U}\rightarrow(j_{V}^{U})_{\ast}(B_{V})$
in $\Shv(\Fet_{U})$.
\end{enumerate}
Here $(j_{V}^{U})_{\ast}(B_{V})(X)=B_{V}(X_{V})$ for $X\in\Fet_{U}$.
These are required to satisfy
\begin{enumerate}
\item the cocycle relation $r_{U}^{W}=(j_{V}^{U})_{\ast}(r_{V}^{W})\circ r_{U}^{V}$
for $W\subset V\subset U$,
\item the sheaf-like condition that for any covering $U=\cup U_{i}$ in
$S_{Zar}$, 
\[
B_{U}=\ker\left(\prod_{i}(j_{U_{i}}^{U})_{\ast}B_{U_{i}}\Rightarrow\prod_{i,j}(j_{U_{i,j}}^{U})_{\ast}B_{U_{i,j}}\right)
\]
in $\Shv(\Fet_{U})$, where $U_{i,j}=U_{i}\cap U_{j}$ as usual. 
\end{enumerate}
A morphism from $((B_{U}),(r_{U}^{V}))$ to $((B_{U}^{\prime}),(r_{U}^{\prime V}))$
is given by a collection $b=(b_{U})_{U}$ of morphisms $b_{U}:B_{U}\rightarrow B_{U}^{\prime}$
in $\Shv(\Fet_{U})$, such that for any $V\subset U$, the diagram
\[
\xyC{4pc}\xymatrix{B_{U}\ar[r]^{b_{U}}\ar[d]_{r_{U}^{V}} & B_{U}^{\prime}\ar[d]^{r_{U}^{\prime V}}\\
(j_{V}^{U})_{\ast}(B_{V})\ar[r]^{(j_{V}^{U})_{\ast}(b_{V})} & (j_{V}^{U})_{\ast}(B_{V}^{\prime})
}
\]
is commutative in $\Shv(\Fet_{U})$. We denote by $\Shv_{\Fet}(S_{Zar})$
the category thus defined. 
\end{defn}

\begin{rem}
We may also describe objects of $\Shv_{\Fet}(S_{Zar})$ as pairs $((B_{U}),(\tilde{r}_{U}^{V}))$
where for $j_{V}^{U}:V\hookrightarrow U$ in $S_{Zar}$, $\tilde{r}_{U}^{V}:(j_{V}^{U})^{\ast}B_{U}\rightarrow B_{V}$
is a morphism in $\Shv(\Fet_{V})$. The cocycle relation becomes $\tilde{r}_{U}^{W}=\tilde{r}_{V}^{W}\circ(j_{W}^{V})^{\ast}(\tilde{r}_{U}^{V})$,
but the formulation of the sheaf-like condition requires passing back
to the adjoint morphisms $B_{U}\rightarrow(j_{V}^{U})_{\ast}B_{V}$.
\end{rem}

\begin{prop}
There is an equivalence of categories 
\[
\gamma_{1}:\Shv(S_{fet})\rightarrow\Shv_{\Fet}(S_{Zar}),\qquad B\mapsto((B_{U}),(r_{U}^{V}))
\]
 where $B_{U}$ is the restriction of $B$ to $\Fet_{U}$ and for
$V\subset U$ and $X\in\Fet_{U}$, 
\[
(r_{U}^{V})_{X}:B_{U}(X)=B(X)\stackrel{\mathrm{res}}{\mathrm{\longrightarrow}}B(X_{V})=B_{V}(X_{V})=(j_{V}^{U})_{\ast}(B_{V})(X).
\]
\end{prop}

\begin{proof}
Any $X$ in $S_{fet}$ belongs to $\Fet_{U}$ where $U$ is the image
of $X$ in $S$, so \textbf{$\gamma_{1}$ is faithful}: for a morphism
$b:B\rightarrow B'$ in $\Shv(S_{fet})$, $b_{X}:B(X)\rightarrow B'(X)$
equals $b_{U,X}:B_{U}(X)\rightarrow B_{U}^{\prime}(X)$. A morphism
$f:Y\rightarrow X$ in $S_{fet}$ factors as $Y\rightarrow X_{V}\hookrightarrow X$
where $V\subset U$ is the image of $Y$ in $S$ and $f':Y\rightarrow X_{V}$
is a morphism in $\Fet_{V}$, so\textbf{ $\gamma_{1}$ is full}: for
any morphism $(b_{U}):\gamma_{1}(B)\rightarrow\gamma_{1}(B')$, the
formula $b_{X}=b_{U,X}$ defines a morphism of presheaves $b:B\rightarrow B'$
since in the commutative diagram
\[
\xyC{3pc}\xymatrix{B(X)\ar[d]_{b_{X}}\ar@{=}[r] & B_{U}(X)\ar[d]_{b_{U,X}}\ar[r]^{(r_{U}^{V})_{X}} & B_{V}(X_{V})\ar[d]^{b_{V,X_{V}}}\ar[r]^{\mathrm{res}_{f'}} & B_{V}(Y)\ar[d]^{b_{V,Y}}\ar@{=}[r] & B(Y)\ar[d]^{b_{Y}}\\
B'(X)\ar@{=}[r] & B_{U}^{\prime}(X)\ar[r]^{(r_{U}^{\prime V})_{X}} & B_{V}^{\prime}(X_{V})\ar[r]^{\mathrm{res}_{f'}} & B_{V}^{\prime}(Y)\ar@{=}[r] & B'(Y)
}
\]
the compositions of the horizontal maps are the restrictions along
$f:Y\rightarrow X$ on $B$ and $B'$; plainly, $\gamma_{1}(b)=(b_{U})$.
Finally, we will see in lemma~\ref{lem:FETCoversRefined} below that
any covering $\{X_{i}\rightarrow X\}$ in $S_{fet}$ has a refinement
$\{X_{j,k}^{\prime}\rightarrow X_{U_{j}}\rightarrow X\}$ where $U=\cup U_{j}$
is a Zariski covering and for each index $j$, $\{X_{j,k}^{\prime}\rightarrow X_{U_{j}}\}$
is a covering in the full subcategory $\Fet_{U_{j}}$. It follows
that \textbf{$\gamma_{1}$ is essentially surjective}. 

Indeed for $((B_{U}),(r_{U}^{V}))$ in $\Shv_{\Fet}(S_{Zar})$, we
may set $B(X)=B_{U}(X)$ and for $f:Y\rightarrow X$, define $\res_{f}:B(X)\rightarrow B(Y)$
by the first line of the above diagram. If $g:Z\rightarrow Y$ is
another morphism in $S_{fet}$, $W\subset V$ the image of $Z$ and
$g':Z\rightarrow Y_{W}$ the induced morphism, the definition of $\Shv_{\Fet}(S_{Zar})$
yields a commutative diagram
\[
\xyC{2pc}\xymatrix{B(X)\ar@{=}[rd]\ar[rrrr]^{\mathrm{res}_{f}} &  &  &  & B(Y)\ar[dddd]^{\mathrm{res}_{g}}\\
 & B_{U}(X)\ar[rd]_{(r_{U}^{W})_{X}}\ar[r]^{(r_{U}^{V})_{X}} & B_{V}(X_{V})\ar[d]^{(r_{V}^{W})_{X_{V}}}\ar[r]^{\mathrm{res}_{f'}} & B_{V}(Y)\ar[d]^{(r_{V}^{W})_{Y}}\ar@{=}[ur]\\
 &  & B_{W}(X_{W})\ar[rd]_{\mathrm{res}_{(f\circ g)^{\prime}}}\ar[r]^{\mathrm{res}_{f'}} & B_{W}(Y_{W})\ar[d]^{\mathrm{res}_{g'}}\\
 &  &  & B_{W}(Z)\ar@{=}[rd]\\
 &  &  &  & B(Z)
}
\]
 whose outer triangle gives $\mathrm{res}_{f\circ g}=\mathrm{res}_{g}\circ\mathrm{res}_{f}$.
We have thus defined a presheaf $B$ on $S_{fet}$, and it remains
to establish that it satisfies the sheaf property with respect to
$(1)$ horizontal coverings $X=\cup X_{U_{i}}$, $U=\cup U_{i}$,
and $(2)$ vertical coverings $\{f_{j}:X_{j}\rightarrow X\}$ entirely
occurring in $\Fet_{U}$. For $(1)$, this follows from the sheaf-like
condition that we have imposed on $((B_{U}),(r_{V}^{U}))$. For $(2)$,
we have to show that 
\[
B(X)\stackrel{?}{=}\ker\left(\prod_{i}B(X_{i})\Rightarrow\prod_{i,j}B(X_{i}\times_{X}X_{j})\right)\quad\text{in}\quad\Set.
\]
Unwinding the definitions, we have to show that 
\[
B_{U}(X)\stackrel{?}{=}\ker\left(\prod_{i}B_{U_{i}}(X_{i})\Rightarrow\prod_{i,j}B_{U_{i,j}^{\prime}}(X_{i}\times_{X}X_{j})\right)\quad\text{in}\quad\Set
\]
where $U_{i}\subset U$ is the image of $X_{i}$ and $U_{i,j}^{\prime}\subset U_{i,j}$
is the image of $X_{i}\times_{X}X_{j}$. Since all schemes in sight
are finite etale over $U$, any $V\in\{U_{i},U_{i,j}^{\prime}\}$
is actually clopen in $U$. Our sheaf-like condition for the resulting
Zariski covering $U=V\coprod V'$ shows that $(r_{U}^{V},r_{U}^{V'})$
induces an isomorphism $B_{U}\simeq(j_{V}^{U})_{\ast}(B_{V})\times(j_{V'}^{U})_{\ast}(B_{V'})$.
Thus
\[
\begin{array}{lrrll}
 & (r_{U}^{U_{i}})_{X_{i}}: & B_{U}(X_{i}) & \rightarrow & B_{U_{i}}(X_{i})\\
\text{and} & (r_{U}^{U_{i,j}^{\prime}})_{X_{i}\times_{X}X_{j}}: & B_{U}(X_{i}\times_{X}X_{j}) & \rightarrow & B_{U_{i,j}^{\prime}}(X_{i}\times_{X}X_{j})
\end{array}
\]
are bijections. The desired equality now becomes the sheaf condition
\[
B_{U}(X)=\ker\left(\prod_{i}B_{U}(X_{i})\Rightarrow\prod_{i,j}B_{U}(X_{i}\times_{X}X_{j})\right)\quad\text{in}\quad\Set
\]
for the sheaf $B_{U}\in\Shv(\Fet_{U})$ with respect to the covering
$\{X_{i}\rightarrow X\}$. So we have constructed a $B\in\Shv(S_{fet})$.
For any open $U$ of $S$ and $Y\in\Fet_{U}$ with image $V\subset U$,
we have $B\vert_{\Fet_{U}}(Y)=B(Y)=B_{V}(Y)$ by definition of $B$,
and we define 
\[
\left(\xymatrix{B_{U}(Y)\ar[r]^{b_{U,Y}} & B\vert_{\Fet_{U}}(Y)}
\right)=\left(\xymatrix{B_{U}(Y)\ar[r]^{(r_{U}^{V})_{Y}} & B_{V}(Y)}
\right).
\]
We have just seen that, as $V$ is clopen in $U$, $b_{U,Y}$ is a
bijection. One checks that $(b_{U,Y})_{Y}$ defines an isomorphism
$b_{U}:B_{U}\rightarrow B\vert_{\Fet_{U}}$ in $\Shv(\Fet_{U})$,
and that $(b_{U})_{U}$ defines an isomorphism $((B_{U}),(r_{U}^{V}))\rightarrow\gamma_{1}(B)$.
So $\gamma_{1}$ is essentially surjective. 
\end{proof}
\begin{lem}
\label{lem:FETCoversRefined}Any covering $\{X_{i}\rightarrow X\}$
in $S_{fet}$ has a refinement of the form 
\[
\{X_{j,k}\rightarrow X_{U_{j}}\hookrightarrow X\}
\]
 where $X_{j,k}$ is open in some $X_{i}$, $U=\cup_{j}U_{j}$ is
a Zariski covering of the image $U$ of $X$ in $S$ and for each
index $j$, $\{X_{j,k}\rightarrow X_{U_{j}}\}$ is a finite covering
in $\Fet_{U_{j}}$. 
\end{lem}

\begin{proof}
Let $\{f_{i}:X_{i}\rightarrow X\}$ be a covering in $S_{fet}$. Let
$\pi:X\rightarrow S$ be the structural morphism and let $U$ be its
image. Since $\pi:X\rightarrow U$ is finite etale, there is an open
partition $U=\coprod_{n\geq1}U_{n}$ such that $X_{n}=\pi^{-1}(U_{n})$
is finite flat of rank $n$ over $U_{n}$. Set $X_{n,i}=\pi_{i}^{-1}(U_{n})$,
where $\pi_{i}:X_{i}\rightarrow S$ is the structural morphism. So
$\{X_{n,i}\rightarrow X_{n}\}$ is a covering in $U_{n,fet}$. If
$\{X_{n,j,k}\rightarrow X_{n,U_{n,j}}\hookrightarrow X_{n}\}$ is
a refinement for $\{X_{n,i}\rightarrow X_{n}\}$ as desired, then
so is $\{X_{n,j,k}\rightarrow X_{U_{n,j}}\hookrightarrow X\}$ for
$\{X_{i}\rightarrow X\}$. We may thus assume that $\pi$ is surjective
of constant degree $n\geq1$. We instead assume that $\pi$ is surjective
of degree \emph{bounded} by $n\geq1$ and argue by induction on $n$.
If $n=1$, then $\pi$ is an isomorphism and the covering $\{X_{i}\twoheadrightarrow f_{i}(X_{i})\hookrightarrow X\}$
is already of the desired form. In general, let $S_{i}$ be the image
of the structural morphism $\pi_{i}:X_{i}\rightarrow S$, so that
$f_{i}$ factors as $X_{i}\rightarrow X_{S_{i}}\hookrightarrow X$,
where $f_{i}^{\prime}:X_{i}\rightarrow X_{S_{i}}$ is a morphism between
finite etale $S_{i}$-schemes, hence itself finite etale. The image
$Y_{i}$ of $f_{i}^{\prime}$ is clopen in $X_{S_{i}}$, so $X_{S_{i}}=Y_{i}\coprod Z_{i}$
with $Y_{i}$ and $Z_{i}$ finite etale over $S_{i}$. The image $S_{i}^{\prime}$
of $Z_{i}$ in $S_{i}$ is clopen in $S_{i}$, so $S_{i}=S_{i}^{\prime}\coprod S_{i,\ast}$.
Since $Y_{i}\rightarrow S_{i}$ is surjective by construction, the
degree of the finite etale surjective morphism $Z_{i}\twoheadrightarrow S_{i}^{\prime}$
is bounded by $n-1$. By our induction hypothesis, each one of the
induced coverings $\{f_{i'}^{-1}(Z_{i})\rightarrow Z_{i}\}$ has a
refinement of the form $\{X_{i,j,k}\rightarrow Z_{i,S_{i,j}}\rightarrow Z_{i}\}$
where $S_{i}^{\prime}=\cup_{j}S_{i,j}$ is a Zariski covering and
$\{X_{i,j,k}\rightarrow Z_{i,S_{i,j}}\}$ is a finite covering in
$\Fet_{S_{i,j}}$, with $X_{i,j,k}$ open in some $f_{i'}^{-1}(Z_{i})\subset X_{i'}$.
For any fixed $i$, $\{S_{i,j}\}\cup\{S_{i,\ast}\}$ is a Zariski
covering of $S_{i}$, and since the $S_{i}$'s cover $S$, we obtain
a Zariski cover of $S$. Over $S_{i,j}\subset S_{i}^{\prime}\subset S_{i}$,
$X_{S_{i,j}}=Y_{i,S_{i,j}}\coprod Z_{i,S_{i,j}}$ has a finite covering
in $\Fet_{S_{i,j}}$ obtained by adjoining to $\{X_{i,j,k}\rightarrow Z_{i,S_{i,j}}\hookrightarrow X_{S_{i,j}}\}$
the single morphism $\{X_{i,S_{i,j}}\twoheadrightarrow Y_{i,S_{i,j}}\hookrightarrow X_{S_{i,j}}\}$.
Over $S_{i,\ast}\subset S_{i}$, $X_{S_{i,\ast}}=Y_{i,S_{i,\ast}}$
is covered by the single morphism $\{X_{i,S_{i,\ast}}\twoheadrightarrow X_{S_{i,\ast}}\}$.
This yields a refinement of $\{X_{i}\rightarrow X\}$ of the desired
form. 
\end{proof}

\section{From Zariski Fet-sheaves to Zariski $\pi$-sheaves}

We now assume that our base scheme $S$ is \emph{irreducible} with
generic point $\eta$, and pick a geometric point $\overline{\eta}$
of $S$ over $\eta$. 

For a nonempty open $U$ of $S$, $U$ is irreducible hence connected,
and the category of sheaves on $\Fet_{U}$ is equivalent to the category
of smooth $\pi(U)$-sets, where the fundamental group $\pi(U)$ is
the automorphism group $\pi(U,\overline{\eta})$ of the fiber functor
$\Fet_{U}\rightarrow\Set$ which takes $X$ to $X(\overline{\eta})$.
Here $\pi(U)$ is equipped with the coarsest topology for which the
actions of $\pi(U)$ on the finite discrete sets $X(\overline{\eta})$
are continuous, so $\pi(U)$ is a profinite group; and given a left
$\pi(U)$-set $Y$, we say that $y\in Y$ is smooth if its stabilizer
$\pi(U)_{y}$ is open in $\pi(U)$, we let $Y^{sm}$ be the set of
smooth points in $Y$, and we say that $Y$ is smooth if $Y^{sm}=Y$,
so that $Y\mapsto Y^{sm}$ is right adjoint to the inclusion of the
strictly full subcategory $\Set_{\pi(U)}^{sm}$ of smooth $\pi(U)$-sets
in the category $\Set_{\pi(U)}$ of all $\pi(U)$-sets. The fiber
functor induces equivalences
\[
\xyC{3pc}\xymatrix{\Fet_{U}\ar[r]^{(-)(\overline{\eta})}\ar@{_{(}->}[d]_{\mathrm{yon}} & \Set_{\pi(U)}^{fsm}\ar@{^{(}->}[d]^{\mathrm{inc}}\\
\Shv(\Fet_{U})\ar[r]^{(-)(\overline{\eta})} & \Set_{\pi(U)}^{sm}
}
\]
where $\Set_{\pi(U)}^{fsm}$ is the full subcategory of finite smooth
$\pi(U)$-sets in $\Set_{\pi(U)}^{sm}$, and the bottom functor takes
a sheaf $B_{U}\in\Shv(\Fet_{U})$ to the smooth $\pi(U)$-set
\begin{align*}
B_{U}(\overline{\eta}) & =\underrightarrow{\lim}_{(X,x)\in\Fet_{U}(\overline{\eta})}B_{U}(X)\\
 & =\underrightarrow{\lim}_{(X,x)\in\Fet_{U}^{c}(\overline{\eta})}B_{U}(X)
\end{align*}
Here $\Fet_{U}(\overline{\eta})$ is the category of pairs $(X,x)$
with $X\in\Fet_{U}$ and $x\in X(\overline{\eta})$, on which $\pi(U)$
acts by $g\cdot(X,x)=(X,g\cdot x)$, and $\Fet_{U}^{c}(\overline{\eta})$
is the $\pi(U)$-stable strictly full initial subcategory where $X$
is connected. 

For nonempty opens $V\subset U$, the base change functor $\Fet_{U}\rightarrow\Fet_{V}$
is compatible with the fiber functors. It induces a continuous morphism
$\pi(V)\rightarrow\pi(U)$, and the pull-back functor $(j_{V}^{U})^{\ast}:\Shv(\Fet_{U})\rightarrow\Shv(\Fet_{V})$
corresponds to the restriction functor $\mathrm{res}:\Set_{\pi(U)}^{sm}\rightarrow\Set_{\pi(V)}^{sm}$.
Accordingly, the data $((B_{U}),(\tilde{r}_{U}^{V}))$ which specifies
a Zariski $\Fet$-sheaf may now be viewed as a pair $(C,\rho)$, where
$C$ is a presheaf of sets on nonempty opens of $S$, with each $C(U)$
equipped with an action $\rho_{U}$ of $\pi(U)$, such that the restriction
maps $C(U)\rightarrow C(V)$ are equivariant with respect to $\pi(V)\rightarrow\pi(U)$.
We extend $C$ to all opens by $C(\emptyset)=\{\star\}$. The sheaf
condition on $((B_{U}),(r_{U}^{V}))$ unwinds to the following sheaf
condition on $C$: for any Zariski covering $U=\cup U_{i}$ of a nonempty
open $U$ of $S$ and any smooth $\pi(U)$-set $Y$, 
\[
\Hom_{\pi(U)}\left(Y,C(U)\right)=\ker\left(\prod_{i}\Hom_{\pi(U_{i})}\left(Y,C(U_{i})\right)\Rightarrow\prod_{i,j}\Hom_{\pi(U_{i,j})}\left(Y,C(U_{i,j})\right)\right)
\]
We denote by $\Shv_{\pi}(S_{Zar})$ the category of these \emph{Zariski
$\pi$-sheaves}, and let
\[
\gamma_{2}:\Shv_{\Fet}(S_{Zar})\rightarrow\Shv_{\pi}(S_{Zar})
\]
be the equivalence of categories just defined.
\begin{rem}
\label{rem:IrreducibilitySImpliesEtaleLocConnec}Irreducibility of
$S$ also implies that any $X$ in $S_{et}$ is locally connected.
Indeed $X\rightarrow S$ is open, hence generizing, so the minimal
points of $X$ belong to $X_{\eta}$. Since $X\rightarrow S$ is etale,
$X_{\eta}$ is discrete, so its points are the minimal points of $X$,
and they are locally finite in $X$. Therefore any quasi-compact open
$U$ of $X$ has finitely many irreducible components, thus also finitely
many connected components; being closed and disjoint, they must be
open in $U$, hence also in $X$. So any point of $X$ has a connected
open neighborhood, and the connected components of $X$ are open.
\end{rem}

\section{From Zariski $\pi$-sheaves to Zariski sheaves of $G$-sets}

With assumptions as above, suppose moreover that $S$ is \emph{geometrically
unibranch}. For $s\in S$, set $S(s)=\Spec(\mathcal{O}_{S,s})$, so
that $S(s)=\underleftarrow{\lim}V$ where $V$ runs through the affine
open neighborhoods of $s$ in $S$. For $U$ open in $S$ with $s\in U$,
pull-back along 
\[
\eta\hookrightarrow S(s)\rightarrow U
\]
and evaluation at $\overline{\eta}$ induce compatible fiber functors
\[
\Fet_{U}\rightarrow\Fet_{S(s)}\rightarrow\Fet_{\eta}\rightarrow\Set
\]
and the corresponding continuous morphisms between the fundamental
groups
\[
\pi(\eta)\rightarrow\pi(S(s))\rightarrow\pi(U).
\]
Our new assumption implies that the latter are surjective \cite[\href{https://stacks.math.columbia.edu/tag/0BQI}{Tag 0BQI}]{SP}.
The group $G=\pi(\eta)$ is the Galois group of the separable closure
of $k(\eta)$ in $k(\overline{\eta})$. We set
\[
I(s)=\ker\left(G\twoheadrightarrow\pi(S(s))\right)\quad\text{and}\quad I(U)=\ker\left(G\twoheadrightarrow\pi(U)\right).
\]
Note that $I(\eta)=\{1\}$, i.e.~$\pi(\eta)\simeq\pi(S(\eta))$ by
\cite[\href{https://stacks.math.columbia.edu/tag/0BQN}{Tag 0BQN}]{SP}.
\begin{prop}
\label{prop:LinkI(s)andI(U)}$(1)$ For any $s\in S$, $I(s)=\cap_{s\in U}I(U)$.
$(2)$ For any nonempty open $U$ of $S$, $I(U)$ is the closure
of the subgroup of $G$ generated by $\{I(s):s\in U\}$. 
\end{prop}

\begin{proof}
$(1)$ Plainly $I(s)\subset\cap_{s\in U}I(U)$. Suppose conversely
that $g\in\cap_{s\in U}I(U)$. Let $X$ be a finite etale $S(s)$-scheme.
By~\cite[8.8.2.ii \& 8.10.5.x]{EGA4.3} and \cite[17.7.8.ii]{EGA4.4},
$X$ extends to a finite etale $U$-scheme $X'$ for some affine open
neighborhood $U$ of $s$ in $S$. Since $g\in I(U)$, $g$ acts trivially
on $X'(\overline{\eta})=X(\overline{\eta})$. Thus $g\in I(s)$. 

$(2)$ Let $I'(U)$ be the closure of the subgroup of $G$ generated
by $\{I(s):s\in U\}$. Plainly $I'(U)\subset I(U)$. Conversely, let
$\Omega$ be an open subgroup of $G$ containing $I'(U)$. Then for
all $s\in U$, $\Omega/I(s)$ is an open subgroup of $\pi(S(s))=G/I(s)$,
so there is a finite connected etale $S(s)$-scheme $X(s)$ and a
point $x(s)\in X(s)(\overline{\eta})$ with stabilizer $\Omega$ in
$G$. As above, we may and do extend $X(s)$ to a scheme $X_{s}$
which is finite etale over some small neighborhood $U_{s}$ of $s$
in $U$. For $s,s'\in U$, there is a unique isomorphism between the
restrictions of $X_{s}$ and $X_{s'}$ over the intersection $U_{s}\cap U_{s'}$
which maps $x(s)$ to $x(s')$. It follows that the $U_{s}$-schemes
$X_{s}$ glue to a scheme $X$ which is finite etale over $U$, and
equipped with a $G$-equivariant isomorphism $G/\Omega\simeq X(\overline{\eta})$.
Since $I(U)$ acts trivially on $X(\overline{\eta})$, we obtain $I(U)\subset\Omega$.
Since $I'(U)$ is closed, it is the intersection of all such $\Omega$'s,
thus also $I(U)\subset I'(U)$. 
\end{proof}
Given the proposition, the following convention seems reasonable:
we set
\[
\pi(\emptyset)=G\quad\text{and}\quad I(\emptyset)=\{1\}.
\]
For any open $U$ of $S$, we identify $\Set_{\pi(U)}$ with the strictly
full subcategory of $\Set_{G}$ where $I(U)$ acts trivially, and
$\Set_{\pi(U)}^{sm}$ with $\Set_{\pi(U)}\cap\Set_{G}^{sm}$. Accordingly,
we may now view a Zariski $\pi$-sheaf as a presheaf of $G$-sets,
and the category $\Shv_{\pi}(S_{Zar})$ as the strictly full subcategory
of $\PreShv_{G}(S_{Zar})$ whose objects are characterized by the
following sheaf-like property: a presheaf of $G$-sets $C$ on $S_{Zar}$
belongs to $\Shv_{\pi}(S_{Zar})$ if and only if for every Zariski
covering $U=\cup_{i}U_{i}$ in $S_{Zar}$, 
\[
C(U)=\ker\left(\prod_{i}C(U_{i})\Rightarrow\prod_{i,j}C(U_{i,j})\right)^{sm,I(U)}.
\]
Note that taking the empty covering of the empty open retrieves our
condition that $C(\emptyset)$ is a singleton, now viewed as a terminal
object in $\Set_{\pi(\emptyset)}^{sm}=\Set_{G}^{sm}$. 

Applying the sheafification functor $a:\PreShv_{G}(S_{Zar})\rightarrow\Shv_{G}(S_{Zar})$,
we obtain
\[
\gamma_{3}:\Shv_{\pi}(S_{Zar})\rightarrow\Shv_{G}(S_{Zar}).
\]

\begin{prop}
\label{prop:FromPiShv2GShv}There is an adjunction 
\[
\gamma_{3}:\Shv_{\pi}(S_{Zar})\longleftrightarrow\Shv_{G}(S_{Zar}):(-)^{sm,I}
\]
\[
\Hom_{\Shv_{G}(S_{Zar})}\left(\gamma_{3}(C),D\right)=\Hom_{\Shv_{\pi}(S_{Zar})}\left(C,D^{sm,I}\right)
\]
for $C\in\Shv_{\pi}(S_{Zar})$ and $D\in\Shv_{G}(S_{Zar})$, where
for any open $U$ of $S$, 
\[
D^{sm,I}(U)=D(U)^{sm,I(U)}
\]
The functor $\gamma_{3}$ is fully faithful, and a Zariski sheaf of
$G$-sets $D\in\Shv_{G}(S_{Zar})$ belongs to the essential image
of $\gamma_{3}$ if and only if the following conditions hold: 
\begin{enumerate}
\item For every quasi-compact open $U$ of $S$, $D(U)$ is a smooth $G$-set.
\item For every $s\in S$, $I(s)$ acts trivially on the stalk $D_{s}$
of $D$ at $s$.
\end{enumerate}
\end{prop}

\begin{proof}
The formula $D^{sm,I}(U)=D(U)^{sm,I(U)}$ defines a subpresheaf $D^{sm,I}$
of $D$ since for opens $V\subset U$ of $S$, we have $I(V)\subset I(U)$,
so $D(U)\rightarrow D(V)$ maps $D^{sm,I}(U)$ to $D(V)^{sm,I(U)}\subset D(V)^{sm,I(V)}=D^{sm,I}(V)$.
If $U=\cup U_{i}$, then 
\begin{align*}
D^{sm,I}(U) & =\ker\left(\prod_{i}D(U_{i})\Rightarrow\prod_{i,j}D(U_{i,j})\right)^{sm,I(U)}\\
 & =\ker\left(\prod_{i}D(U_{i})^{sm,I(U_{i})}\Rightarrow\prod_{i,j}D(U_{i,j})^{sm,I(U_{i,j})}\right)^{sm,I(U)}\\
 & =\ker\left(\prod_{i}D^{sm,I}(U_{i})\Rightarrow\prod_{i,j}D^{sm,I}(U_{i,j})\right)^{sm,I(U)}
\end{align*}
using the sheaf property of $D$ for the first equality, so $D^{sm,I}$
belongs to $\Shv_{\pi}(S_{Zar})$ and the functor $(-)^{sm,I}:\Shv_{G}(S_{Zar})\rightarrow\Shv_{\pi}(S_{Zar})$
is well-defined. We have
\[
\begin{array}{rcl}
\Hom_{\Shv_{G}(S_{Zar})}\left(\gamma_{3}(C),D\right) & = & \Hom_{\PreShv_{G}(S_{Zar})}\left(C,D\right)\\
 & = & \Hom_{\PreShv_{G}(S_{Zar})}\left(C,D^{sm,I}\right)\\
 & = & \Hom_{\Shv_{\pi}(S_{Zar})}\left(C,D^{sm,I}\right)
\end{array}
\]
by the defining adjunction for $\gamma_{3}$, the fact that all $C(U)$'s
are smooth $G$-sets fixed by $I(U)$, and the fact that $\Shv_{\pi}(S_{Zar})$
is a full subcategory of $\PreShv_{G}(S_{Zar})$. 

The full faithfulness of $\gamma_{3}$ is equivalent to the unit $C\rightarrow\gamma_{3}(C)^{sm,I}$
being an isomorphism, which we now establish. It is a monomorphism
since the sheaf-like condition on $C$ implies that $C$ is a separated
presheaf, so already $C\rightarrow\gamma_{3}(C)$ is a monomorphism.
For any section $d$ of $D=\gamma_{3}(C)$ over $U$, there is a Zariski
covering $U=\cup U_{i}$ such that $d\vert_{U_{i}}=c_{i}$ in $D(U_{i})$
for some $c_{i}\in C(U_{i})$; then $(c_{i})$ belongs to the kernel
$K$ of $\prod_{i}C(U_{i})\Rightarrow\prod_{i,j}C(U_{i,j})$ by injectivity
of $C(U_{i,j})\rightarrow D(U_{i,j})$. If $d$ moreover belongs to
$D^{sm,I}(U)$, there is a compact open subgroup $\Omega$ of $G$
containing $I(U)$ and fixing $d$; then $\Omega$ also fixes all
$d\vert_{U_{i}}$'s, all $c_{i}$'s by injectivity of $C(U_{i})\rightarrow D(U_{i})$,
and $(c_{i})$ thus belongs to $K^{sm,I(U)}$, which is the image
of $C(U)$ by the sheaf-like property of $C$; so there's a $c\in C(U)$
such that $c\vert_{U_{i}}=c_{i}=d\vert_{U_{i}}$ for all $i$, whence
$d=c$ by the sheaf property of $D$. So $C(U)=D^{sm,I}(U)$ as claimed. 

A sheaf of $G$-sets $D$ belongs to the essential image of $\gamma_{3}$
if and only if the counit $c_{D}:\gamma_{3}(D^{sm,I})\rightarrow D$
is an isomorphism. Note that $c_{D}$ is always a monomorphism since
$D^{sm,I}$ is a subpresheaf of $D$, so $c_{D}$ is an isomorphism
if and only if it is an epimorphism, i.e.~for any $U$ and $d\in D(U)$,
there is a covering $U=\cup U_{i}$ such that the stabilizer $G_{d\vert_{U_{i}}}$
of $d\vert_{U_{i}}$ is open in $G$ and contains $I(U_{i})$. The
necessity of our conditions is thus clear. Conversely, assume $(1)$
and $(2)$. Fix $s\in S$, $d_{s}\in D_{s}$, and lift $d_{s}$ to
$d\in D(U)$ for some affine neighborhood $U$ of $s$. Let $G_{d}$
and $G_{d_{s}}$ be the stabilizers of $d$ and $d_{s}$, so that
$G_{d}\subset G_{d_{s}}$. By $(1)$, $G_{d}$ is open in $G$, so
$[G_{d_{s}}:G_{d}]$ is finite. Pick a set of representatives $\sigma_{1},\cdots,\sigma_{n}\in G_{d_{s}}$
for $G_{d_{s}}/G_{d}$. Since 
\[
(\sigma_{1}d)_{s}=\cdots=(\sigma_{n}d)_{s},
\]
there is a smaller affine neighborhood $s\in U'\subset U$ where also
\[
\sigma_{1}d\vert_{U'}=\cdots=\sigma_{n}d\vert_{U'}.
\]
Replacing $U$ by $U'$, we may thus assume that $G_{d_{s}}=G_{d}$.
The $G$-orbit of $d$ is finite and for $s\in V_{1}\subset V_{2}\subset U$,
we have $I(V_{1})\cdot d\subset I(V_{2})\cdot d$. So shrinking $U$
again, we may assume that $I(V)\cdot d=I(U)\cdot d$ for all $V\subset U$
with $s\in V$. This means that 
\[
I(V)/I(V)_{d}\simeq I(U)/I(U)_{d}
\]
or equivalently:
\[
I(U)_{d}/I(V)_{d}\simeq I(U)/I(V)
\]
for all $V\subset U$ with $s\in V$. Taking limits over all such
$V'$s, we find that 
\[
I(U)_{d}/I(s)_{d}\simeq I(U)/I(s)
\]
by proposition~\ref{prop:LinkI(s)andI(U)}, or equivalently, 
\[
I(s)/I(s)_{d}\simeq I(U)/I(U)_{d}.
\]
But by $(2)$, $I(s)\subset G_{d_{s}}$ which equals $G_{d}$, so
$I(s)_{d}=I(s)$, and $I(U)_{d}=I(U)$. In other words, $d\in D^{sm,I}(U)$.
A fortiori, $d_{s}$ belongs to the image of 
\[
c_{D,s}:\gamma_{3}(D^{sm,I})_{s}\rightarrow D_{s}.
\]
Since $d_{s}$ and $s$ were arbitrary, $c_{D}$ is surjective on
stalks, hence an epimorphism, thus an isomorphism: $D\simeq\gamma_{3}(D^{sm,I})$
belongs to the essential image of $\gamma_{3}$. 
\end{proof}
We denote by $\Shv_{G}^{\star}(S_{Zar})$ the strictly full subcategory
of $\Shv_{G}(S_{Zar})$ defined in the previous proposition, so that
$\gamma_{3}$ induces an equivalence of categories
\[
\gamma_{3}:\Shv_{\pi}(S_{Zar})\rightarrow\Shv_{G}^{\star}(S_{Zar}).
\]

\begin{rem}
\label{rem:IrreducibilityUnibranchSImpliesEtalesLocIrred}If $S$
is irreducible and geometrically unibranch, any $X$ in $S_{et}$
is locally irreducible. Indeed we have seen that any quasi-compact
open $U$ of $X$ has finitely many irreducible components. Since
$U$ is etale over $S$, it is also geometrically unibranch, so its
irreducible components are disjoint, hence open. It follows that the
irreducible components of $X$ are open, and equal to its connected
components.
\end{rem}

\section{The category $\Shv_{G}^{\star}(S_{Zar})$}

With assumptions and notations as in the previous section ($S$ irreducible
and geometrically unibranch), we review here various constructions
related to Zariski sheaves of $G$-sets. We write $\Shv$, $\Shv_{G}^{\star}$
etc\ldots{} for $\Shv(S_{Zar})$, $\Shv_{G}^{\star}(S_{Zar})$ etc\ldots{}

\subsection{Sheafification}

We start with the classical adjunction
\begin{equation}
a:\PreShv_{G}\longleftrightarrow\Shv_{G}:\mathrm{inc}
\end{equation}
where $a$ is the sheafification functor. For $D\in\PreShv_{G}$,
the sheaf of sets underlying the sheaf of $G$-sets $a(D)$ is the
sheafification of the presheaf of sets underlying $D$. For every
$s\in S$, the unit $D\rightarrow a(D)$ induces an isomorphism of
$G$-sets $D_{s}\rightarrow a(D)_{s}$. 

\subsection{Smoothness for sheaves of $G$-sets.}

Let $D$ be a sheaf of $G$-sets. For a subgroup $H$ of $G$, the
formula $D^{H}(U)=D(U)^{H}$ defines a subsheaf $D^{H}$ of $D$,
with $(D^{H})_{s}\subset(D_{s})^{H}$ for all $s\in S$. We denote
by $D^{sm}$ the subsheaf of $D$ defined by 
\[
D^{sm}=\cup D^{\Omega}\quad\text{in}\quad\Shv
\]
where $\Omega$ runs through the compact open subgroups of $G$; this
is a $G$-stable subsheaf of $D$ with $D(U)^{sm}\subset D^{sm}(U)$
for every open $U$ of $S$ and $(D^{sm})_{s}\subset(D_{s})^{sm}$
for every $s\in S$. We say that $D$ is smooth if $D^{sm}=D$, and
we denote by $\Shv_{G}^{sm}$ the strictly full subcategory of smooth
sheaves of $G$-sets in $\Shv_{G}$. We have adjunctions:
\begin{equation}
\mathrm{inc}:\Shv_{G}^{sm}\longleftrightarrow\Shv_{G}:(-)^{sm}.
\end{equation}
The counit is the embedding $D^{sm}\hookrightarrow D$. 
\begin{lem}
For a Zariski sheaf of $G$-sets $D$, the following are equivalent:
\begin{enumerate}
\item For every quasi-compact open $U$ of $S$, $D(U)$ is a smooth $G$-set,
\item For every affine open $U$ of $S$, $D(U)$ is a smooth $G$-set,
\item $D$ is smooth, i.e.~$D\in\Shv_{G}^{sm}$. 
\end{enumerate}
They imply 
\begin{enumerate}[resume]
\item For every $s\in S$, $D_{s}$ is a smooth $G$-set.
\item For every open $U$ of $S$ and $d\in D(U)$ the stabilizer $G_{d}$
of $d$ is closed in $G$.
\end{enumerate}
\end{lem}

\begin{proof}
Plainly $(1)\Rightarrow(2)\Rightarrow(4)$ and $(4)\Rightarrow(5)$
by injectivity of $D(U)\rightarrow\prod_{s\in U}D_{s}$. If $(2)$
holds, then $D^{sm}(U)=D(U)$ for all affine open $U$, hence for
all $U$ by the sheaf property for $D^{sm}$ and $D$, so $(2)\Rightarrow(3)$.
Finally $(3)\Rightarrow(1)$ by \cite[\href{https://stacks.math.columbia.edu/tag/0738}{Tag 0738}]{SP}. 
\end{proof}

\subsection{The $\star$-condition}

Plainly, $\Shv_{G}^{\star}\subset\Shv_{G}^{sm}$ and again, there
is an adjunction 
\begin{equation}
\mathrm{inc}:\Shv_{G}^{\star}\longleftrightarrow\Shv_{G}^{sm}:(-)^{\star}.
\end{equation}
The right adjoint functor takes $D\in\Shv_{G}^{sm}$ to its subsheaf
$D^{\star}$ defined by
\[
D^{\star}(U)=\left\{ d\in D(U):\forall s\in U,\,d_{s}\in D_{s}^{I(s)}\right\} 
\]
The counit is the embedding $D^{\star}\hookrightarrow D$ and 
\[
\forall s\in S:\qquad D_{s}^{\star}=D_{s}^{I(s)}.
\]
Indeed $D_{s}^{\star}\subset D_{s}^{I(s)}$ by definition, and the
proof of proposition~\ref{prop:FromPiShv2GShv} gives the other inclusion:
the stabilizer $G_{d_{s}}$ of $d_{s}\in D_{s}^{I(s)}$ is open in
$G$ and contains $I(s)=\cap_{s\in U}I(U)$, so it contains $I(U)$
for some sufficiently small affine neighborhood $U$ of $s$. Shrinking
$U$ if necessary, we may assume that $d_{s}$ lifts to $d\in D(U)$,
and shrinking it further, that $G_{d}=G_{d_{s}}$. Then $G_{d}$ contains
$I(U)$, so $d\in D^{\star}(U)$, hence $d_{s}\in D_{s}^{\star}$. 

\subsection{Monomorphisms and epimorphisms\protect\label{subsec:MonoEpisInShv*}}

Let $d$ be a morphism in $\Shv_{G}^{\star}$. 

Since $\Shv_{G}^{\star}$ is stable under fiber products in $\PreShv_{G}$,
$d$ is a monomorphism in $\Shv_{G}^{\star}$ if and only if $d$
is a monomorphism in $\PreShv_{G}$. Since $\Shv_{G}^{\star}\hookrightarrow\Shv_{G}$
has a right adjoint whose counits are monomorphisms, $d$ is an epimorphism
in $\Shv_{G}^{\star}$ if and only if $d$ is an epimorphism in $\Shv_{G}$
(i.e.~also: an epimorphism in $\Shv$). 

\subsection{Sheaves of $G$-sets with trivial $G$-action\protect\label{subsec:SheavesNoAction}}

They belong to $\Shv_{G}^{\star}$, and form a strictly full subcategory
which we identify with $\Shv$. It fits into adjunctions 
\begin{equation}
\mathrm{inc}:\Shv\longleftrightarrow\Shv_{G}^{\star}:(-)^{G}
\end{equation}
\begin{equation}
(-)_{G}:\Shv_{G}^{\star}\longleftrightarrow\Shv:\mathrm{inc}\label{eq:Adj*NoAct}
\end{equation}
These are both restrictions of analogous adjunctions for the embedding
of $\Shv$ into the larger category $\Shv_{G}$, where the right and
left adjoints are given by
\[
D^{G}(U)=D(U)^{G}\quad\text{and}\quad D_{G}=a(U\mapsto G\backslash D(U)).
\]
The counit $D^{G}\hookrightarrow D$ is a monomorphism and the unit
$D\twoheadrightarrow D_{G}$ is an epimorphism. 

\subsection{\protect\label{subsec:ConstantSheaves}Constant sheaves}

The constant sheaf functor has left and right adjoints, 
\begin{equation}
(-)_{\eta}:\Shv_{G}\longleftrightarrow\Set_{G}:(-)_{S}
\end{equation}
\begin{equation}
(-)_{S}:\Set_{G}\longleftrightarrow\Shv_{G}:\Gamma(S,-)
\end{equation}
Here $\Gamma(S,D)=D(S)$ while $(-)_{\eta}$ is the stalk functor
at the generic point $\eta$ of $S$. The unit $X\rightarrow\Gamma(S,X_{S})$
and counit $(X_{S})_{\eta}\rightarrow X$ are isomorphisms, so $(-)_{S}$
is fully faithful. Passing to the smooth subcategories, we obtain
adjunctions
\begin{equation}
(-)_{\eta}:\Shv_{G}^{sm}\longleftrightarrow\Set_{G}^{sm}:(-)_{S}
\end{equation}
\begin{equation}
(-)_{S}:\Set_{G}^{sm}\longleftrightarrow\Shv_{G}^{sm}:\Gamma(S,-)^{sm}.
\end{equation}
Composing $(-)_{S}$ with $(-)^{\star}$ gives a new functor $(-)_{S}^{\star}$
fitting in an adjunction
\begin{equation}
(-)_{\eta}:\Shv_{G}^{\star}\longleftrightarrow\Set_{G}^{sm}:(-)_{S}^{\star}\label{eq:Adj**S}
\end{equation}
with unit $D\rightarrow(D_{\eta})_{S}^{\star}$. The counit is an
isomorphism since $(X_{S}^{\star})_{\eta}=X^{I(\eta)}=X$, so the
right adjoint functor $(-)_{S}^{\star}$ is fully faithful. 
\begin{lem}
\label{lem:EssImageOf*S}For $D\in\Shv_{G}^{\star}$, the following
conditions are equivalent:
\begin{enumerate}
\item $D$ belongs to the essential image of $(-)_{S}^{\star}$
\item The unit $D\rightarrow(D_{\eta})_{S}^{\star}$ is an isomorphism.
\item For any $s\in S$, $D_{s}\rightarrow D_{\eta}^{I(s)}$ is bijective,
\item For any open $U\neq\emptyset$ of $S$, $D(U)\rightarrow D_{\eta}^{I(U)}$
is bijective.
\end{enumerate}
\end{lem}

\begin{proof}
$(1)\Leftrightarrow(2)$ by general properties of adjunctions, $(2)\Leftrightarrow(3)$
since 
\[
\forall s\in S:\quad((D_{\eta})_{S}^{\star})_{s}=D_{\eta}^{I(s)}
\]
and $(2)\Leftrightarrow(4)$ since for $U\neq\emptyset$, 
\[
(D_{\eta})_{S}^{\star}(U)=\{d\in D_{\eta}:\forall s\in U,\,d\in D_{\eta}^{I(s)}\}=D_{\eta}^{I(U)}
\]
by proposition~\ref{prop:LinkI(s)andI(U)} and smoothness of $D_{\eta}$.
\end{proof}

\subsection{Et/Set sheaves\protect\label{subsec:Et/Set-sheaves}}

For $D\in\Shv_{G}$, taking images of counits in $\Shv_{G}$, we define
\[
D_{et}=\mathrm{Im}\left(D\rightarrow(D_{\eta})_{S}\times D_{G}\right)\quad\text{and}\quad D_{set}=\mathrm{Im}\left(D\rightarrow(D_{\eta})_{S}\right).
\]
For $D\in\Shv_{G}^{\star}$, these are also images of the corresponding
counits in $\Shv_{G}^{\star}$:
\begin{align*}
D_{set} & =\mathrm{Im}\left(D\rightarrow(D_{\eta})_{S}^{\star}\right)\\
D_{et} & =\mathrm{Im}\left(D\rightarrow(D_{\eta})_{S}^{\star}\times D_{G}\right)\\
 & =\mathrm{Im}\left(D\rightarrow D_{set}\times D_{G}\right)
\end{align*}
So we have epimorphisms in $\Shv_{G}^{\star}$, 
\[
\xyR{1ex}\xyC{2pc}\xymatrix{ &  & D_{set}\\
D\ar@{->>}[r] & D_{et}\ar@{->>}[ur]\ar@{->>}[dr]\\
 &  & D_{G}
}
\]
We denote by $\Shv_{G}^{et}$ and $\Shv_{G}^{set}\subset\Shv_{G}^{et}$
the strictly full subcategories of $\Shv_{G}^{\star}$ where $D\simeq D_{et}$,
respectively $D\simeq D_{set}$. We thus obtain new adjunctions 
\begin{equation}
(-)_{et}:\Shv_{G}^{\star}\longleftrightarrow\Shv_{G}^{et}:\mathrm{inc}\label{eq:AdjStarEt}
\end{equation}
\begin{equation}
(-)_{set}:\Shv_{G}^{\star}\longleftrightarrow\Shv_{G}^{set}:\mathrm{inc}\label{eq:AdjStarSet}
\end{equation}
with epimorphic units $D\twoheadrightarrow D_{et}$ and $D\twoheadrightarrow D_{set}$. 

\subsection{Opens}

For an open embedding $j_{U}:U\hookrightarrow S$, we have the usual
adjunctions
\begin{equation}
j_{U}^{\ast}:\Shv_{G}(S_{Zar})\longleftrightarrow\Shv_{G}(U_{Zar}):j_{U*}
\end{equation}
\begin{equation}
j_{U!}:\Shv_{G}(U_{Zar})\longleftrightarrow\Shv_{G}(S_{Zar}):j_{U}^{\ast}\label{eq:AdjjU!jU*}
\end{equation}
where $j_{U!}$ takes $E\in\Shv_{G}(U_{Zar})$ to its extension
\[
V\in S_{Zar}\mapsto j_{U!}(E)(V)=\begin{cases}
\emptyset & \text{if }V\not\subset U,\\
E(V) & \text{if }V\subset U.
\end{cases}
\]
The unit $E\rightarrow j_{U}^{\ast}j_{U!}E$ is an isomorphism, so
$j_{U!}$ is fully faithful, and the counit $j_{U!}j_{U}^{\ast}D\rightarrow D$
is a monomorphism. The adjunction (\ref{eq:AdjjU!jU*}) preserves
the full subcategories $\Shv_{G}^{sm}$, $\Shv_{G}^{\star}$, $\Shv$,
and is compatible with the stalk functor $(-)_{\eta}$ (if $U\neq\emptyset$),
so it is compatible with all of the above constructions -- details
left to the reader. 

\subsection{Characterization of et/set-sheaves\protect\label{subsec:CharactOfEtSetSheaves}}

For every sheaf of $G$-sets $D\in\Shv_{G}$, we also have the following
constructions. 
\begin{itemize}
\item The support of $D$ is the open subset of $S$ defined by
\[
\Supp(D)=\left\{ s\in S:D_{s}\neq\emptyset\right\} .
\]
\item The set of \emph{connected components} of $D$ is $\pi_{0}(D)=G\backslash D_{\eta}$.
For $c\in\pi_{0}(D)$ viewed as a $G$-orbit in $D_{\eta}$, we define
$D(c)$ by the cartesian diagram 
\[
\xyR{2pc}\xymatrix{D(c)\ar[r]\ar@{^{(}->}[d] & c_{S}\ar@{^{(}->}[d]\\
D\ar[r] & (D_{\eta})_{S}
}
\]
So $D(c)$ is a $G$-stable subsheaf of $D$ and $D=\coprod_{c\in\pi_{0}(D)}D(c)$
in $\Shv_{G}$.
\item We denote by $\mathscr{S}_{G}(D)$ the set of all pairs $(U,\gamma)$
where $U\neq\emptyset$ is open in $S$ and $\gamma\subset D(U)$
is a $G$-orbit, and equip $\mathscr{S}_{G}(D)$ with the partial
order 
\[
(U,\gamma)\leq(U',\gamma')\iff U\subset U'\quad\text{and}\quad\gamma=\gamma'\vert_{U}.
\]
We denote by $\mathscr{S}_{G}^{m}(D)$ the set of maximal elements
in $\left(\mathscr{S}_{G}(D),\leq\right)$. 
\item Sending $(U,\gamma)\in\mathscr{S}_{G}(D)$ to the image $\gamma_{\eta}\subset D_{\eta}$
of $\gamma\subset D(U)$ defines maps 
\[
\mathscr{S}_{G}^{m}(D)\hookrightarrow\mathscr{S}_{G}(D)\twoheadrightarrow\pi_{0}(D).
\]
We denote by $\mathscr{S}_{G}^{m}(D,c)\hookrightarrow\mathscr{S}_{G}(D,c)$
their fibers over $c\in\pi_{0}(D)$. Thus
\[
\mathscr{S}_{G}^{m}(D,c)=\mathscr{S}_{G}^{m}(D(c))\quad\text{and}\quad\mathscr{S}_{G}(D,c)=\mathscr{S}_{G}(D(c)).
\]
\item Fix $(U,\gamma)\in\mathscr{S}_{G}(D)$. Since $D(U)=\Gamma(U,j_{U}^{\ast}D)$,
the isomorphisms 
\begin{align*}
\Hom_{G}\left(\gamma,D(U)\right) & =\Hom_{\Shv_{G}(U_{Zar})}\left(\gamma_{U},j_{U}^{\ast}D\right)\\
 & =\Hom_{\Shv_{G}(S_{Zar})}\left(j_{U!}(\gamma_{U}),D\right)
\end{align*}
map the $G$-equivariant embedding $\gamma\hookrightarrow D(U)$ to
a morphism 
\[
j_{U!}(\gamma_{U})\rightarrow D\quad\text{in}\quad\Shv_{G}(S_{Zar})
\]
whose evaluation at $V\subset S$ is $\emptyset\rightarrow D(V)$
if $V\not\subset U$, and the composition of the embedding $\gamma\hookrightarrow D(U)$
with the restriction map $D(U)\rightarrow D(V)$ if $V\subset U$.
In particular it factors through $D(c)\subset D$, where $c=\gamma_{\eta}\in\pi_{0}(D)$.
\item Summing these morphisms across $\mathscr{S}_{G}^{m}(D)$, we obtain
a morphism
\[
\coprod_{(U,\gamma)\in\mathscr{S}_{G}^{m}(D)}j_{U!}(\gamma_{U})\rightarrow D\quad\text{in}\quad\Shv_{G}(S_{Zar})
\]
whose fiber over $D(c)\hookrightarrow D$ is the morphism
\[
\coprod_{(U,\gamma)\in\mathscr{S}_{G}^{m}(D(c))}j_{U!}(\gamma_{U})\rightarrow D(c)\quad\text{in}\quad\Shv_{G}(S_{Zar}).
\]
\end{itemize}
\begin{lem}
If $D$ is smooth, any element of $\mathscr{S}_{G}(D)$ has a majorant
in $\mathscr{S}_{G}^{m}(D)$.
\end{lem}

\begin{proof}
Fix $(U,\gamma)$. By Zorn's lemma, we have to show that any chain
$\mathscr{C}$ in 
\[
\mathscr{M}=\{(V,\theta)\in\mathscr{S}_{G}(D):(U,\gamma)\leq(V,\theta)\}
\]
has an upper bound in $\mathscr{M}$. Let $U'=\cup_{(V,\theta)\in\mathscr{C}}V$.
Then any element of 
\[
\underleftarrow{\lim}_{(V,\theta)\in\mathscr{C}}\theta\subset\underleftarrow{\lim}_{(V,\theta)\in\mathscr{C}}D(V)=D(U')
\]
defines a $G$-orbit $\gamma'\subset D(U')$ such that $(U',\gamma')\in\mathscr{M}$
dominates $\mathscr{C}$. So we have to show that the left hand side
limit is not empty. For each single $(V,\theta)\in\mathscr{C}$, we
may first choose some $x\in\theta$, consider the corresponding orbit
map $G\twoheadrightarrow\theta$, and equip $\theta$ with the induced
quotient topology. One checks that it does not depend upon $x$, and
turns $\theta$ into a compact topological space which is Hausdorff
since the stabilizer $G_{x}$ is closed. Then $(V,\theta)\mapsto\theta$
is an inverse system of nonempty compact Hausdorff spaces indexed
by a filtered set (namely $\mathscr{C}$), so its limit is not empty. 
\end{proof}
\begin{prop}
\label{prop:CaractOfsetAndet}Suppose that $D\in\Shv_{G}^{sm}(S_{Zar})$.
Then 
\begin{enumerate}
\item The morphism 
\[
\coprod_{(U,\gamma)\in\mathscr{S}_{G}^{m}(D)}j_{U!}(\gamma_{U})\rightarrow D
\]
 is an epimorphism.
\item The canonical map 
\[
\mathscr{S}_{G}^{m}(D)\rightarrow\pi_{0}(D)
\]
 is surjective. 
\item The following conditions are equivalent:
\begin{enumerate}
\item For any opens $\emptyset\neq V\subset U$ of $S$, $D(U)\rightarrow D(V)$
is injective on orbits,
\item For any open $U$ of $S$ and $s\in U$, $D(U)\rightarrow D_{s}$
is injective on orbits,
\item For any open $U\neq\emptyset$ of $S$, $D(U)\rightarrow D_{\eta}$
is injective on orbits,
\item For any $s\in S$, $D_{s}\rightarrow D_{\eta}$ is injective on orbits,
\item For any specialization $s^{\prime}\rightsquigarrow s$ in $S$, $D_{s}\rightarrow D_{s'}$
is injective on orbits,
\item For every $(U,\gamma)\in\mathscr{S}_{G}(D)$, $j_{U!}(\gamma_{U})\rightarrow D$
is a monomorphism,
\item For every $(U,\gamma)\in\mathscr{S}_{G}^{m}(D)$, $j_{U!}(\gamma_{U})\rightarrow D$
is a monomorphism,
\item The product of units $D\rightarrow(D_{\eta})_{S}\times D_{G}$ is
a monomorphism,
\item The unit $D\rightarrow D_{et}$ is an isomorphism.
\end{enumerate}
\item The following conditions are equivalent:
\begin{enumerate}
\item For any opens $\emptyset\neq V\subset U$ of $S$, $D(U)\rightarrow D(V)$
is injective,
\item For any open $U$ of $S$ and $s\in U$, $D(U)\rightarrow D_{s}$
is injective,
\item For any open $U\neq\emptyset$ of $S$, $D(U)\rightarrow D_{\eta}$
is injective,
\item For any $s\in S$, $D_{s}\rightarrow D_{\eta}$ is injective,
\item For any specialization $s^{\prime}\rightsquigarrow s$ in $S$, $D_{s}\rightarrow D_{s'}$
is injective,
\item $D$ satisfies the conditions of $(3)$ and $\mathscr{S}_{G}^{m}(D)\rightarrow\pi_{0}(D)$
is bijective, 
\item The morphism $\coprod_{(U,\gamma)\in\mathscr{S}_{G}^{m}(D)}j_{U!}(\gamma_{U})\rightarrow D$
is an isomorphism,
\item The unit $D\rightarrow(D_{\eta})_{S}$ is a monomorphism.
\item The unit $D\rightarrow D_{set}$ is an isomorphism.
\end{enumerate}
\end{enumerate}
\end{prop}

\begin{proof}
$(1)$ For $U\subset S$ and $d\in D(U)$ with $G$-orbit $\gamma$,
pick $(U',\gamma')\in\mathscr{S}_{G}^{m}(D)$ over $(U,\gamma)$.
Evaluating $j_{U'!}\gamma_{U'}^{\prime}\rightarrow D$ at $U$ gives
$\gamma^{\prime}\hookrightarrow D(U')\rightarrow D(U)$, whose image
equals $\gamma$, and so contains $d$: our morphism is already surjective
in $\PreShv$.

$(2)$ Let $c$ be a $G$-orbit in $D_{\eta}$, $d_{\eta}\in c$,
$U$ a small open of $S$ where $d_{\eta}$ lifts to $d\in D(U)$,
$\gamma$ the $G$-orbit of $d$, and $(U',\gamma')\in\mathscr{S}_{G}^{m}(D)$
over $(U,\gamma)$. Then $(U',\gamma')$ maps to $\gamma_{\eta}=c$
in $G\backslash D_{\eta}=\pi_{0}(D)$: our morphism $\mathscr{S}_{G}^{m}(D)\rightarrow\pi_{0}(D)$
is surjective.

In $(3)$ and $(4)$, the equivalence of conditions $(a)$ trough
$(e)$ are easy, $(d)\Leftrightarrow(h)$ since monomorphicity is
equivalent to injectivity on all stalks, and $(h)\Leftrightarrow(i)$
by definition of $D_{et}$ and $D_{set}$. Also: $(3,a)\Leftrightarrow(3,f)\Rightarrow(3,g)$
and $(4,g)\Rightarrow(4,a)$ are obvious. It remains to establish
the following three implications: 

$(3,g)\Rightarrow(3,a)$. Let $V\subset U$ be nonempty opens of $S$,
$\gamma$ a $G$-orbit in $D(U)$. Pick $(U',\gamma')\in\mathscr{S}_{G}^{m}(D)$
above $(U,\gamma)$. By $(3,g)$, $j_{U'!}\gamma_{U'}^{\prime}\rightarrow D$
is a monomorphism. Evaluating it at $V$ gives an injection $\gamma^{\prime}\hookrightarrow D(U')\rightarrow D(V)$,
which factors as 
\[
\gamma'\twoheadrightarrow\gamma\hookrightarrow D(U)\rightarrow D(V)
\]
So $\gamma'\rightarrow\gamma$ is a bijection and the restriction
map $D(U)\rightarrow D(V)$ is injective on $\gamma$. 

$(4,c)\Rightarrow(4,f)$. Plainly $(4,c)\Rightarrow(3,c)$, so $D$
satisfies all conditions of $(3)$. Suppose that $(U_{1},\gamma_{1})$
and $(U_{2},\gamma_{2})$ in $\mathscr{S}_{G}^{m}(D)$ have the same
image $\gamma_{\eta}$ in $\pi_{0}(D)=G\backslash D_{\eta}$. Fix
$x\in\gamma_{\eta}$ and lift it to $x_{i}\in\gamma_{i}$. By $(4,c)$
applied to $V=U_{1}\cap U_{2}\neq\emptyset$, $x_{1}\vert_{V}=x_{2}\vert_{V}$
in $D(V)$, so the $x_{i}$'s glue to $x\in D(U)$ where $U=U_{1}\cup U_{2}$.
Let $\gamma$ be the $G$-orbit of $x$. Then $(U_{i},\gamma_{i})\leq(U,\gamma)$
in $\mathscr{S}_{G}(D)$, whence $(U_{1},\gamma_{1})=(U,\gamma)=(U_{2},\gamma_{2})$
by maximality of $(U_{i},\gamma_{i})$. Given $(2)$, this proves
$(4,f)$.

$(4,f)\Rightarrow(4,g)$. Using $(1)$ and the decomposition $D=\coprod_{c\in\pi_{0}(D)}D(c)$,
it is sufficient to establish that for any $c\in\pi_{0}(D)$, the
morphism
\[
\coprod_{(U,\gamma)\in\mathscr{S}_{G}^{m}(D,c)}j_{U!}(\gamma_{U})\rightarrow D(c)
\]
is a monomorphism. The second part of $(4,f)$ tells us that $\mathscr{S}_{G}^{m}(D,c)$
contains a single element $(U,\gamma)$, and the first part of $(4,f)$
tells us that $(3,f)$ holds, which implies that $j_{U!}(\gamma_{U})\rightarrow D(c)$
is indeed a monomorphism. 
\end{proof}
\begin{rem}
If $D\in\Shv_{G}^{\star}$, we may replace $(D_{\eta})_{S}$ by $(D_{\eta})_{S}^{\star}$
in $(3,h)$ and $(4,h)$, and
\[
(3,i)\iff D\in\Shv_{G}^{et},\qquad(4,i)\iff D\in\Shv_{G}^{set}.
\]
\end{rem}

\begin{lem}
\label{lem:PropertiesOfet}For $D\in\Shv_{G}^{et}(S_{Zar})$ and \emph{any}
open $U$ of $S$, 
\[
D(U)\text{ is a smooth }G\text{-set with trivial action of }I(U).
\]
\end{lem}

\begin{proof}
If $U=\emptyset$, $D(U)=\star$ with trivial action. If $U\neq\emptyset$
and $s\in U$, the localization $D(U)\rightarrow D_{s}$ is $G$-equivariant,
injective on orbits, and $D_{s}$ is smooth, with trivial action of
$I(s)$. It follows that $D(U)$ is smooth, with trivial action of
$I(s)$ for all $s\in U$, hence with trivial action of $I(U)$ by
proposition~\ref{prop:LinkI(s)andI(U)}. 
\end{proof}

\section{Harvest}

\subsection{~}

We now assume all of the above assumptions:
\[
S\text{\,is locally henselian, geometrically unibranch, and irreducible.}
\]
We thus have equivalence of categories
\begin{equation}
\xyR{1pt}\xymatrix{\AlgSp_{et}(S)\ar[r]^{\alpha} & \Shv(S_{et})\ar[r]\sp(0.47){\beta} & \Shv(S_{fet})\ar[r]\sp(0.45){\gamma} & \Shv_{G}^{\star}(S_{Zar})\\
A\ar@{|->}[r] & B\ar@{|->}[r] & C\ar@{|->}[r] & D
}
\label{eq:ABCDdiagram}
\end{equation}
where $\gamma=\gamma_{3}\circ\gamma_{2}\circ\gamma_{1}$ and $G=\pi_{1}(\eta,\overline{\eta})$.
We set $\delta=\gamma\circ\beta\circ\alpha$. In the sequel, we may
often simplify our notations to $\AlgSp_{et}=\AlgSp_{et}(S)$, $\Shv_{G}^{\star}=\Shv_{G}^{\star}(S_{Zar})$,
etc\ldots{}

\subsection{\protect\label{subsec:MonoEpiInAlgSpet}Monomorphisms and epimorphisms}

Let $a:A_{1}\rightarrow A_{2}$ be a morphism in $\AlgSp_{et}$ with
image $d:D_{1}\rightarrow D_{2}$ in $\Shv_{G}^{\star}$. Since $\delta$
is an equivalence of categories
\[
\begin{array}{rcl}
a\text{\,is a monomorphism} & \iff & d\text{ is a monomorphism}\\
a\text{ is an epimorphism} & \iff & d\text{ is an epimorphism}
\end{array}
\]
We have described monomorphisms and epimorphisms of $\Shv_{G}^{\star}$
in section~\ref{subsec:MonoEpisInShv*}.

In $\AlgSp_{et}$, monomorphisms are open immersions. Indeed since
the embedding of $\AlgSp_{et}$ into $\AlgSp$ commutes with fiber
products, $a$ is a monomorphism in $\AlgSp_{et}$ if and only if
it is a monomorphism in $\AlgSp$. Since $A_{1}$ and $A_{2}$ are
etale over $S$, $a$ is an etale morphism by~\cite[\href{https://stacks.math.columbia.edu/tag/03FV}{Tag 03FV}]{SP},
so $a$ is a monomorphism if and only if it is an open immersion by~\cite[\href{https://stacks.math.columbia.edu/tag/05W5}{Tag 05W5}]{SP}.
We thus obtain:
\[
a\text{ is an open immersion}\iff d\text{ is a monomorphism of presheaves.}
\]

In $\AlgSp_{et}$, $a$ is an epimorphism if and only if $a$ is surjective.
Indeed by~\cite[\href{https://stacks.math.columbia.edu/tag/03MF}{Tag 03MF}]{SP},
$a$ is surjective if and only if there is a commutative diagram
\[
\xyR{2pc}\xymatrix{A_{1}^{\prime}\ar[r]\ar[d] & A_{2}^{\prime}\ar[d]\\
A_{1}\ar[r] & A_{2}
}
\]
with $A_{i}^{\prime}\in S_{et}$, $A_{1}^{\prime}\rightarrow A_{2}^{\prime}$
surjective, and $A_{1}^{\prime}\rightarrow A_{1}$ (representable)
surjective and etale. As observed in the proof of proposition~\ref{prop:FromAlgSp2EtShv},
$A_{i}^{\prime}\rightarrow A_{i}$ and $A_{1}^{\prime}\rightarrow A_{2}^{\prime}$
are then epimorphisms in $\Shv((\Sch/S)_{et}))$, thus also in $\Shv((\Sch/S)_{fppf})$,
$\AlgSp(S)$ and $\AlgSp_{et}(S)$, so $a$ is an epimorphism. Conversely
if $a$ is an epimorphism in $\AlgSp_{et}$, let $A_{2}^{\prime}\rightarrow A_{2}$
and $A_{1}^{\prime}\rightarrow A_{1}\times_{A_{2}}A_{2}^{\prime}$
be surjective etale morphisms with $A_{i}^{\prime}\in S_{et}$. Then
$A_{1}^{\prime}\rightarrow A_{1}$ is also surjective etale, and we
want to show that $A_{1}^{\prime}\rightarrow A_{2}^{\prime}$ is surjective.
We have just seen that $A_{1}^{\prime}\rightarrow A_{1}\times_{A_{2}}A_{2}^{\prime}$
is an epimorphism in $\AlgSp_{et}$; since $a$ is an epimorphism,
so is its base change $A_{1}\times_{A_{2}}A'_{2}\rightarrow A_{2}^{\prime}$
in the topos $\AlgSp_{et}(S)\simeq\Shv(S_{et})$, thus $A_{1}^{\prime}\rightarrow A_{2}^{\prime}$
is also an epimorphism in $\AlgSp_{et}$. We are reduced to: if a
morphism $a:A_{1}\rightarrow A_{2}$ is $S_{et}$ is an epimorphism
in $\AlgSp_{et}$, it is surjective as a morphism of schemes. This
is obvious: glue two copies of $A_{2}$ along the (open!) image of
$a$ to obtain a new scheme $A_{3}\in S_{et}$, equipped with two
morphisms $b_{1},b_{2}:A_{2}\rightarrow A_{3}$ such that $b_{1}\circ a=b_{2}\circ a$
in $S_{et}\subset\AlgSp_{et}$. Since $a$ is an epimorphism in $\AlgSp_{et}$,
$b_{1}=b_{2}$, therefore $a$ is surjective. So:
\[
a\text{ is surjective}\iff d\text{ is an epimorphism of sheaves.}
\]

\subsection{Representable etale algebraic spaces}
\begin{lem}
\label{lem:DeltaOnFet}For $X\in S_{fet}$ with image $U$ in $S$
viewed as an element of $\AlgSp_{et}$, 
\[
\delta(X)=j_{U!}(X(\overline{\eta})_{U})\quad\text{in}\quad\Shv_{G}.
\]
\end{lem}

\begin{proof}
Set $A=\Hom_{S}(-,X)$ and $D=\delta(A)$. Unwinding the definitions,
$D$ is the Zariski sheafification of the presheaf which takes $V\in S_{Zar}$
to the $G$-set with trivial action of $I(V)$ corresponding to the
smooth $\pi(V)$-set defined by
\begin{align*}
B_{V}(\overline{\eta}) & =\underrightarrow{\lim}_{(Y,y)\in\Fet_{V}(\overline{\eta})}B_{V}(Y)=\underrightarrow{\lim}_{(Y,y)\in\Fet_{V}(\overline{\eta})}\Hom_{V}(Y,X_{V}).
\end{align*}
Since $S$ is irreducible, so are the nonempty $V$'s, thus for any
$(Y,y)\in\Fet_{V}(\overline{\eta})$, the nonempty clopen image of
the finite etale map $Y\rightarrow V$ equals $V$, i.e.~$Y\twoheadrightarrow V$
is surjective. Hence $B_{V}(\overline{\eta})=\emptyset$ when $V$
is not contained in the image $U$ of $X$ in $S$. For $V\subset U$,
$X_{V}\in\Fet_{V}$ and the map which takes $s\in\Hom_{V}(Y,X_{V})$
to $s(y)\in X(\overline{\eta})$ induces a $G$-equivariant isomorphism
$B_{V}(\overline{\eta})\rightarrow X(\overline{\eta})$ of smooth
$G$-sets. It follows that our presheaf is already a sheaf, namely
$D=j_{U!}(X(\overline{\eta})_{U})$. 
\end{proof}
\begin{cor}
There is a $2$-commutative diagram 
\[
\xyR{1.5pc}\xyC{2pc}\xymatrix{\Fet_{U}\ar[r]\sp(0.45){(-)(\overline{\eta})}\ar@{_{(}->}[d] & \Set_{\pi(U)}^{fsm}\ar[r]\sp(0.45){(-)_{U}} & \Shv_{G}^{set}(U)\ar[d]^{j_{U!}}\\
S_{fet}\ar@{_{(}->}[d]\ar[rr]^{\delta} &  & \Shv_{G}^{set}(S)\ar@{^{(}->}[d]\\
\AlgSp_{et}(S)\ar[rr]^{\delta} &  & \Shv_{G}^{\star}(S)
}
\]
\end{cor}

\begin{proof}
For $X\in\Fet_{U}$, $Y=X(\overline{\eta})$ belongs to $\Set_{\pi(U)}^{fsm}$;
for $Y\in\Set_{\pi(U)}^{sm}$, the constant $G$-sheaf $Z=Y_{U}$
belongs to $\Shv_{G}^{set}(U)$; for $Z\in\Shv_{G}^{set}(U)$, the
extension $D=j_{U!}(Z)$ belongs to $\Shv_{G}^{set}(S)$. For $X,Y,Z,D$
as indicated, we have seen that $D=\delta(X)$. This gives the $2$-commutativity
of the outer diagram. Since any $X\in S_{fet}$ belongs to $\Fet_{U}$
for some $U$ (the image of $X$ in $S$), we also have $\delta(S_{fet})\subset\Shv_{G}^{set}(S)$.
\end{proof}
\begin{thm}
\label{thm:Main}The functor $\delta$ induces equivalence of categories
\[
\xyR{1.5pc}\xyC{3pc}\xymatrix{S_{set}\ar[r]^{\delta}\ar@{_{(}->}[d]_{\mathrm{yon}} & \Shv_{G}^{set}\ar@{^{(}->}[d]^{\mathrm{inc}}\\
S_{et}\ar[r]^{\delta}\ar@{_{(}->}[d]_{\mathrm{yon}} & \Shv_{G}^{et}\ar@{^{(}->}[d]^{\mathrm{inc}}\\
\AlgSp_{et}\ar[r]^{\delta} & \Shv_{G}^{\star}
}
\]
where $S_{set}$ is the strictly full subcategory of separated $S$-schemes
in $S_{et}.$
\end{thm}

\begin{proof}
Let $A,B,C,D$ be as in \ref{eq:ABCDdiagram}. We have
\begin{align*}
 & A\in\AlgSp_{et}(S)\text{ is representable (by some }X\in S_{et})\\
\iff & B\in\Shv(S_{et})\text{ is representable (by the same \ensuremath{X\in S_{et}})}\\
\iff & C=\cup C_{i}\text{ in }\Shv(S_{fet})\text{\,with }C_{i}\,\text{representable (by some }X_{i}\in S_{fet})\\
\iff & D=\cup D_{i}\text{ in }\Shv_{G}^{\star}(S_{Zar})\,\text{with }D_{i}\simeq\delta(X_{i})\text{ for some }X_{i}\in S_{fet}
\end{align*}
using remark~\ref{rem:RepAlgSpIsRepEtale} and corollary~\ref{cor:FetShevRepbyEt}.
We have seen that $\delta$ maps $S_{fet}$ to $\Shv_{G}^{set}\subset\Shv_{G}^{et}$.
Since $\Shv_{G}^{et}$ is stable under arbitrary unions, we thus find
that $\delta$ maps representable $A$'s to $D$'s in $\Shv_{G}^{et}$.
Suppose conversely that $D\in\Shv_{G}^{et}$. Then by proposition~\ref{prop:CaractOfsetAndet},
\[
D=\cup_{(U,\gamma)\in\mathscr{S}_{G}^{m}(D)}D_{(U,\gamma)}\quad\text{in}\quad\Shv_{G}^{\star}
\]
where $D_{(U,\gamma)}=j_{U!}(\gamma_{U})$ with $\gamma\in\Set_{\pi(U)}^{fsm}$
by lemma~\ref{lem:PropertiesOfet}, whence 
\[
D_{(U,\gamma)}\simeq\delta(X_{(U,\gamma)})
\]
 for an $S$-scheme $X_{(U,\gamma)}\in\Fet_{U}\subset S_{fet}$ with
$X_{(U,\gamma)}(\overline{\eta})\simeq\gamma$ as $G$-sets. Thus
$D$ satisfies the last displayed property, and $A$ is representable
by some $X\in S_{et}$ which is a union of the $X_{(U,\gamma)}$'s.
If moreover $D\in\Shv_{G}^{set}$, then in fact 
\[
D={\textstyle \coprod_{(U,\gamma)\in\mathscr{S}_{G}^{m}(D)}}D_{(U,\gamma)}\quad\text{in}\quad\Shv_{G}^{\star}
\]
so $A$ is representable by $X=\coprod_{(U,\gamma)}X_{(U,\gamma)}$,
which is indeed separated over $S$. 

Finally, suppose that $A=\Hom_{S}(-,X)$ for some $X\in S_{set}$.
We have to show that $D\in\Shv_{G}^{set}$, and we use the characterization~$(4,c)$
of proposition~\ref{prop:CaractOfsetAndet}. So let $U\neq\emptyset$
be open in $S$, and suppose that $d_{1},d_{2}\in D(U)$ have the
same image $d_{1,\eta}=d_{2,\eta}=d$ in $D_{\eta}$. Let $\gamma_{i}$
be the $G$-orbit of $d_{i}$, $\gamma$ the $G$-orbit of $d$. We
already know that $D\in\Shv_{G}^{et}$, thus $D(U)\rightarrow D_{\eta}$
induces $G$-equivariant bijections $\gamma_{1}\rightarrow\gamma$
and $\gamma_{2}\rightarrow\gamma$, whose inverse we denote by $\gamma\ni e\mapsto e_{i}\in\gamma_{i}$.
By lemma~\ref{lem:PropertiesOfet}, $\gamma\in\Set_{\pi(U)}^{fsm}$,
so there is a $(Y,y)\in\Fet_{U}^{c}(\overline{\eta})$ with $(Y(\overline{\eta}),y)\simeq(\gamma,d)$
as pointed $G$-sets. Since $\delta(Y)\simeq j_{U!}(\gamma_{U})$,
we obtain morphisms $j_{i}:Y\rightarrow X$, corresponding to the
morphisms $j_{U!}(\gamma_{U})\simeq j_{U!}(\gamma_{i,U})\hookrightarrow D$.
Let $Z\hookrightarrow Y$ be the pull-back of the diagonal of $X$
under $(j_{1},j_{2}):Y\rightarrow X\times_{S}X$ and set $E=\delta(Z)$.
Since $\delta$ commutes with fiber products, $E$ is the pull-back
of the diagonal of $D$ under $j_{U!}(\gamma)\rightarrow D\times D$.
So 
\[
E(U)=\{e\in\gamma:e_{1}=e_{2}\text{\,in }D(U)\}\quad\text{and}\quad E_{\eta}=\gamma.
\]
Since $X$ is separated etale over $S$, $Z\hookrightarrow Y$ is
a clopen immersion, and since $Y$ is finite etale over $U$, so is
$Z$. Thus $E\simeq j_{U!}(Z(\overline{\eta})_{U})$ and $E(U)\rightarrow E_{\eta}$
is a bijection. It follows that $d_{1}=d_{2}$, which checks the characterization
$(4,c)$, so $D\in\Shv_{G}^{set}$.
\end{proof}
\begin{rem}
A punctual scheme $S$ is irreducible, locally henselian and geometrically
unibranch, with $\Shv_{G}^{\star}=\Shv_{G}^{et}=\Shv_{G}^{set}$.
So we retrieve the well known fact that over such an $S$, every etale
algebraic space is representable by a separated $S$-scheme. 
\end{rem}

\begin{cor}
\label{cor:SetisDisjointUnionOfFet}Any $X\in S_{set}$ is a disjoint
union of irreducible fet $S$-schemes.
\end{cor}

\begin{proof}
This either follows from the proof of theorem~\ref{thm:Main}, or
from the equivalence $\delta:S_{set}\rightarrow\Shv_{G}^{set}$ and
the characterization $(4,g)$ of $\Shv_{G}^{set}$ in proposition~\ref{prop:CaractOfsetAndet}.
\end{proof}

\subsection{Underlying topological spaces}

Let $\Top$ be the category of topological spaces and continuous maps.
For $\mathcal{S}\in\Top$, we denote by $\mathcal{S}_{loc}$ the strictly
full subcategory of $\Top/\mathcal{S}$ whose objects are local homeomorphisms
$\mathcal{X}\rightarrow\mathcal{S}$. It is well known that the functor
$\mathrm{sec}:\mathcal{S}_{loc}\rightarrow\Shv(\mathcal{S})$ which
takes $\mathcal{X}$ to its sheaf of sections is an equivalence of
categories. An inverse functor maps a sheaf of sets $\mathcal{C}$
on $\mathcal{S}$ to the corresponding \emph{espace étalé}: the underlying
set is $\coprod_{s\in\mathcal{S}}\mathcal{C}_{s}$, and a basis of
open neighborhoods of $x\in\mathcal{C}_{s}$ is given by the subsets
$\{c_{y}:y\in V\}$, for $V$ open in $S$ containing $s$ and $c\in\mathcal{C}(V)$
with $c_{s}=x$. 

Likewise, if $S_{loc}$ is the full subcategory of $S_{et}$ whose
objects are local isomorphisms $X\rightarrow S$, then the functor
$\mathrm{sec}:S_{loc}\rightarrow\Shv(S_{Zar})$ which takes $X\rightarrow S$
to its sheaf of sections is an equivalence of categories. An inverse
functor maps a sheaf of sets $C$ on $S_{Zar}$ to the object $X$
of $S_{loc}$ whose underlying topological space $\left|X\right|$
is as above, equipped with the sheaf of rings $\mathcal{O}_{X}$ which
is the pull-back through $\left|X\right|\rightarrow\left|S\right|$
of $\mathcal{O}_{S}$. It follows that the functor $S_{loc}\rightarrow\left|S\right|_{loc}$
which forgets the structure sheaves is an equivalence of categories. 
\begin{prop}
There is a $2$-commutative diagram of equivalences
\[
\xyR{2pc}\xymatrix{S_{loc}\ar@{_{(}->}[d]_{\mathrm{yon}}\ar[r]^{\mathrm{sec}} & \Shv\ar@{^{(}->}[d]^{\mathrm{inc}}\\
\AlgSp_{et}\ar[r]^{\delta} & \Shv_{G}^{\star}
}
\]
\end{prop}

\begin{proof}
Let $f:X\rightarrow S$ be a local isomorphism, $\mathcal{X}=\mathrm{sec}(X)$
its sheaf of sections, $D=\delta(X)$. Since $X\in S_{loc}$, there
is a Zariski covering $X=\cup X_{i}$ such that $f$ induces an isomorphism
from $X_{i}$ to an open $U_{i}=f(X_{i})$ of $S$. Then $D=\cup D_{i}$
with 
\[
D_{i}=\delta(X_{i})=j_{U_{i}!}(\star_{U_{i}})\quad\text{where}\quad X_{i}(\overline{\eta})=\{\star\}.
\]
In particular $D_{i}$ has trivial action of $G$, and so does $D=\cup D_{i}$.
Let $U$ be an open of $S$. A section $x\in\mathcal{X}(U)$ is an
$S$-morphism $U\rightarrow X$ which $\delta$ maps to a morphism
$j_{U!}(\star_{U})\rightarrow D$ corresponding to a section $d\in D(U)$.
Conversely, a section $d\in D(U)$ gives a morphism $j_{U!}(\star_{U})\rightarrow D$
which is the image of an $S$-morphism $U\rightarrow X$ which is
a section $x\in\mathcal{X}(U)$. We thus obtain an isomorphism of
sheaves $\mathcal{X}\rightarrow D$ on $S_{Zar}$ which is plainly
functorial in $X$, and this yields the desired isomorphism 
\[
\mathrm{inc}\circ\mathrm{sec}\rightarrow\delta\circ\mathrm{yon}
\]
 between functors $S_{loc}\rightarrow\Shv_{G}^{\star}$. 
\end{proof}
We have seen in \ref{subsec:SheavesNoAction} that the embedding $\Shv\hookrightarrow\Shv_{G}^{\star}$
has left and right adjoints 
\[
(-)_{G},(-)^{G}:\Shv_{G}^{\star}\rightarrow\Shv
\]
with epimorphic units $D\twoheadrightarrow D_{G}$ and monomorphic
counits $D^{G}\hookrightarrow D$. So the Yoneda embedding $S_{loc}\rightarrow\AlgSp_{et}$
has matching left and right adjoints
\[
(-)_{loc},(-)^{loc}:\AlgSp_{et}\rightarrow S_{loc}
\]
 with surjective units $A\twoheadrightarrow A_{loc}$ and open immersion
counits $A^{loc}\hookrightarrow A$. To describe these functors, we
first record the following consequence of our assumptions.
\begin{lem}
The topological space $\left|A\right|$ associated to $A\in\AlgSp_{et}$
belongs to $\left|S\right|_{loc}$.
\end{lem}

\begin{proof}
Suppose first that $A=\Hom_{S}(-,X)$ for some $X\in S_{et}$. Then
$\left|A\right|=\left|X\right|$ by \cite[\href{https://stacks.math.columbia.edu/tag/03BX}{Tag 03BX}]{SP}.
By proposition~\ref{prop:IfLocHensEtcovbyFet}, $X$ has a Zariski
covering $X=\cup X_{i}$ with $X_{i}\in S_{fet}$. By remark~\ref{rem:IrreducibilityUnibranchSImpliesEtalesLocIrred},
we may assume that $X_{i}$ is irreducible. Then by proposition~\ref{prop:IntegralOverLocHensIsHomeo},
$X_{i}\rightarrow S$ is an homeomorphism onto its image. So $\left|X\right|\in\left|S\right|_{loc}$.
For the general case, let $X\rightarrow A$ be a surjective etale
morphism with $X\in S_{et}$. Then $R=X\times_{A}X$ also belongs
to $S_{et}$, and $\left|A\right|$ is the quotient of $\left|X\right|$
by the equivalence relation $\left|R\right|$ in $\Top$, see~\cite[\href{https://stacks.math.columbia.edu/tag/03BX}{Tag 03BX}]{SP}.
Since $\left|X\right|,\left|R\right|\in\left|S\right|_{loc}\simeq\Shv(\left|S\right|)$,
there is also a quotient $Q=\left|X\right|/\left|R\right|$ in $\left|S\right|_{loc}$.
Since morphisms in $\left|S\right|_{loc}$ are local homeomorphisms,
hence open, $\left|X\right|\rightarrow Q$ is open. Since $\left|X\right|\rightarrow Q$
is an epimorphism in $\left|S\right|_{loc}$, it is surjective: glue
two copies of $Q$ along the open image of $\left|X\right|$ to obtain
$Q^{\prime}$ in $\left|S\right|_{loc}$ with two continuous maps
$Q\rightarrow Q'$ inducing the same map $\left|X\right|\rightarrow Q'$;
then these two maps must be equal, so $\left|X\right|\rightarrow Q$
is indeed surjective. Since open surjective continuous maps are quotient
maps in $\Top$, $Q$ is the quotient of $\left|X\right|$ by the
equivalence relation $\left|X\right|\times_{Q}\left|X\right|$ in
$\Top$. Since this fiber product belongs to $\left|S\right|_{loc}$,
it is equal to the corresponding fiber product in $\left|S\right|_{loc}\simeq\Shv(\left|S\right|)$,
which is $\left|R\right|$ by general properties of topoi. Thus $\left|A\right|\simeq Q$,
and $\left|A\right|$ belongs to $\left|S\right|_{loc}$.
\end{proof}
\begin{prop}
Left and right adjoints of $S_{loc}\rightarrow\AlgSp_{et}$ are as
follows: 
\begin{enumerate}
\item $(-)_{loc}:\AlgSp_{et}\rightarrow S_{loc}$ takes $A$ to the $S$-scheme
$A_{loc}\in S_{loc}$ with 
\[
\left|A_{loc}\right|=\left|A\right|,\qquad\mathcal{O}_{A_{loc}}=\text{pull-back of }\mathcal{O}_{S}\text{\,through }\left|A\right|\rightarrow\left|S\right|.
\]
The unit $A\twoheadrightarrow A_{loc}$ is a surjective homeomorphism.
\item $(-)^{loc}:\AlgSp_{et}\rightarrow S_{loc}$ takes $A$ to the largest
open subspace $A^{loc}$ of $A$ which belongs to $S_{loc}$. The
counit is the open embedding $A^{loc}\hookrightarrow A$.
\end{enumerate}
\end{prop}

\begin{proof}
The given formula for $A_{loc}$ defines a functor 
\[
(-)_{loc}:\AlgSp_{et}\rightarrow S_{loc}.
\]
We have to show that it is left adjoint to the embedding $S_{loc}\hookrightarrow\AlgSp_{et}$.
This amounts to constructing a functorial unit $A\rightarrow A_{loc}$
which is an isomorphism when $A\in S_{loc}$. If $A=\Hom_{S}(-,X)$
for some $f:X\rightarrow S$ in $S_{et}$, $X_{loc}$ has the same
underlying space as $X$, with structure sheaf $\mathcal{O}_{X_{loc}}=f^{\ast}\mathcal{O}_{S}$
(pull-back as sheaves). The structure morphism $X\rightarrow S$ gives
a morphism $\mathcal{O}_{S}\rightarrow f_{\ast}\mathcal{O}_{X}$ whose
adjoint $f^{\ast}\mathcal{O}_{S}\rightarrow\mathcal{O}_{X}$ is a
morphism $\mathcal{O}_{X_{loc}}\rightarrow\mathrm{Id}_{\ast}\mathcal{O}_{X}$,
which yields our unit $\epsilon_{X}:X\rightarrow X_{loc}$. This construction
is functorial in $X$ and gives the identity when $X\in S_{loc}$,
so it produces the desired left adjoint of $S_{loc}\hookrightarrow S_{et}$.
Extending our units from $S_{et}$ to $\AlgSp_{et}$ is then a formal
consequence of proposition~\ref{prop:FromAlgSp2EtShv}: to define
$\epsilon_{A}:A\rightarrow A_{loc}$ in $\AlgSp_{et}$ amounts to
defining a morphism between the corresponding sheaves on $S_{et}$,
and this we do by mapping a section $a\in A(X)$ corresponding to
a morphism $a:X\rightarrow A$, to the section $a_{loc}^{\sim}\in A_{loc}(X)$
which corresponds to the morphism
\[
\xymatrix{X\ar[r]^{\epsilon_{X}} & X_{loc}\ar[r]^{a_{loc}} & A_{loc}.}
\]
For the right adjoint, we know that it exists, and its counits are
open immersions. Thus $A^{loc}$ is an open subspace of $A$ belonging
to $S_{loc}$, and by the adjunction property, any such open is contained
in $A^{loc}$. So $A^{loc}$ is the largest such open. 
\end{proof}

\subsection{Connected components}

By definition, the set of connected components $\pi_{0}(A)$ of an
algebraic space $A$ over $S$ is the set of connected components
of the associated topological space $\left|A\right|$. For $A\in\AlgSp_{et}$,
$\left|A\right|=\left|A_{loc}\right|$, so $\pi_{0}(A)$ is also the
set of connected components of the etale $S$-scheme $A_{loc}$, which
by remark~\ref{rem:IrreducibilitySImpliesEtaleLocConnec}, is the
set of points of its generic fiber $A_{loc,\eta}$. Since $A_{loc}\rightarrow S$
is a local isomorphism, so is its generic fiber, thus $\pi_{0}(A)$
is also the set of sections of $A_{loc,\eta}\rightarrow\eta$. If
$D=\delta(A)$, then $D_{G}\simeq\delta(A_{loc})\simeq\mathrm{sec}(A_{loc})$
so $\pi_{0}(A)\simeq D_{G,\eta}\simeq G\backslash D_{\eta}=\pi_{0}(D)$. 

We may also work this out as follows, without passing through $A_{loc}$.
Since $A$ is etale over $S$, $\left|A\right|$ is locally connected:
this follows from remark~\ref{rem:IrreducibilitySImpliesEtaleLocConnec}
and \cite[\href{https://stacks.math.columbia.edu/tag/03BT}{Tag 03BT}]{SP}
(and holds whenever $S$ locally has a finite number of irreducible
components). We may then redefine $\pi_{0}(A)$ as the set of minimal
nonempty clopen subsets of $\left|A\right|$, and $\left|A\right|$
is the disjoint union of these connected components. By~\cite[\href{https://stacks.math.columbia.edu/tag/03BZ}{Tag 03BZ}]{SP},
opens of $\left|A\right|$ correspond bijectively with opens of $A$,
which as seen above, are just the subobjects of $A$ in $\AlgSp_{et}$.
Thus $\pi_{0}(A)$ is also the set of minimal nontrivial complemented
elements in the poset of all subobjects of $A$ in $\AlgSp_{et}$,
and $A$ is the disjoint union of these subobjects. Passing this categorical
description through the equivalence $\delta$, we obtain another set
$\pi_{0}^{\prime}(D)$ of connected components of $D=\delta(A)$,
and we must check that it matches the set $\pi_{0}(D)$ from section~\ref{subsec:CharactOfEtSetSheaves}.
Since
\[
D={\textstyle \coprod_{c\in\pi_{0}(D)}}D(c)\quad\text{in}\quad\Shv_{G}^{\star}
\]
we have to show that any nontrivial complemented $G$-stable subsheaf
$D_{1}$ of $D$ contained in $D(c)$ equals $D(c)$. Since $D_{1}$
is complemented in $D$, it is also complemented in $D(c)$: there
is a $G$-stable subsheaf $D_{2}$ with $D(c)=D_{1}\coprod D_{2}$.
But then $c=D_{1,\eta}\coprod D_{2,\eta}$, and this forces one of
the $D_{i,\eta}$'s to be empty, which can only occur if $D_{i}$
itself is the empty sheaf. So $D_{1}=D(c)$, and indeed $\pi_{0}(D)=\pi_{0}^{\prime}(D)$. 

\subsection{Base change}

Let $S'$ be another irreducible, locally henselian and geometrically
unibranch scheme, and let $f:S^{\prime}\rightarrow S$ be a morphism
inducing an isomorphism on generic points. For instance, we could
take any of the following: 
\begin{enumerate}
\item An open immersion $j_{U}:U\hookrightarrow S$ with $U\neq\emptyset$,
\item A monomorphism $S(s)\hookrightarrow S$ with $S(s)=\Spec(\mathcal{O}_{S,s})$
for some $s\in S$,
\item The closed immersion $S_{red}\hookrightarrow S$,
\item The generic point map $\iota:\eta\rightarrow S$.
\end{enumerate}
\begin{prop}
\label{prop:BaseChange}There is a $2$-commutative diagram
\[
\xyC{3pc}\xymatrix{\AlgSp_{et}(S)\ar[d]_{f^{\ast}}\ar[r]^{\alpha} & \Shv(S_{et})\ar[d]^{f^{\ast}}\ar[r]\sp(0.45){\gamma\circ\beta} & \Shv_{G}^{\star}(S_{Zar})\ar[d]^{f^{\ast}}\ar@{^{(}->}[r]^{\mathrm{inc}} & \Shv_{G}(S_{Zar})\ar[d]^{f^{\ast}}\\
\AlgSp_{et}(S')\ar[r]^{\alpha} & \Shv(S_{et}^{\prime})\ar[r]\sp(0.45){\gamma\circ\beta} & \Shv_{G}^{\star}(S_{Zar}^{\prime})\ar@{^{(}->}[r]^{\mathrm{inc}} & \Shv_{G}(S_{Zar}^{\prime})
}
\]
\end{prop}

\begin{proof}
The $2$-commutativity of the first square was established in proposition~\ref{prop:BaseChange4alpha}.
For the second and third square, we first construct a base change
morphism 
\[
\xyR{1.2pc}\xyC{1.5pc}\xymatrix{\Shv\left(S_{et}\right)\ar[dd]_{f^{\ast}}\ar[rr]\sp(0.45){\gamma\circ\beta} &  & \Shv_{G}(S_{Zar})\ar[dd]^{f^{\ast}}\\
 &  & \,\ar@{..>}@/_{1pc}/[dl]_{bc}\\
\Shv\left(S_{et}^{\prime}\right)\ar[rr]\sb(0.45){\gamma\circ\beta} & \vphantom{\,m} & \Shv_{G}(S_{Zar}^{\prime})
}
\quad\text{adjoint to}\quad\xymatrix{\Shv\left(S_{et}\right)\ar[dd]_{f^{\ast}}\ar[rr]\sp(0.45){\gamma\circ\beta} & \vphantom{\,m}\,\ar@{..>}@/_{1pc}/[dr]_{bc'} & \Shv_{G}(S_{Zar})\\
 &  & \,\\
\Shv\left(S_{et}^{\prime}\right)\ar[rr]\sb(0.45){\gamma\circ\beta} &  & \Shv_{G}(S_{Zar}^{\prime})\ar[uu]_{f_{\ast}}
}
\]
So let $B\in\Shv(S_{et})$, $B'=f^{\ast}B$ and let $C$ and $C'$
be their images under $\gamma_{2}\circ\gamma_{1}\circ\beta$, viewed
as presheaves of $G$-sets, so that $\gamma_{3}$ is the sheafification
functor $a$. We have to construct a morphism of sheaves $bc':a(C)\rightarrow f_{\ast}a(C')$,
or equivalently a morphism of presheaves $C\rightarrow f_{\ast}a(C')$;
it is sufficient to construct a morphism of presheaves $C\rightarrow f_{\ast}C'$.
Now for any open $U$ of $S$ with preimage $U'=f^{-1}(U)$ in $S'$,
we have 
\begin{align*}
C(U) & =\underrightarrow{\lim}_{(X,x)\in\Fet_{U}(\overline{\eta})}B(X)\\
f_{\ast}C^{\prime}(U) & =\underrightarrow{\lim}_{(X',x')\in\Fet_{U'}(\overline{\eta})}B'(X')
\end{align*}
The unit $B\rightarrow f_{\ast}f^{\ast}B=f_{\ast}B'$ induces a morphism
\[
\xyR{0.1pc}\xyC{2pc}\xymatrix{\Fet_{U}(\overline{\eta})\ar[dddd]_{-\times_{U}U'}\ar[ddrr]\sp(0.5){B(-)}\\
 & \,\ar@/_{.5pc}/@{..>}[dd]\\
 &  & \Set\\
 & \,\\
\Fet_{U'}(\overline{\eta})\ar[uurr]\sb(0.5){B'(-)}
}
\]
whose colimit gives a $G$-equivariant map $C(U)\rightarrow f_{\ast}C'(U)$,
which yields the desired morphism $C\rightarrow f_{\ast}C'$. As a
consequence, we obtain a factorization
\[
\xyR{1.2pc}\xyC{1.6pc}\xymatrix{\Shv\left(S_{et}\right)\ar[dd]_{f^{\ast}}\ar[rr]\sp(0.5){\gamma\circ\beta} &  & \Shv_{G}^{\star}(S_{Zar})\ar[dd]^{f^{\ast}}\ar@{^{(}->}[rr]\sp(0.5){\mathrm{inc}} &  & \Shv_{G}(S_{Zar})\ar[dd]^{f^{\ast}}\\
 &  & \,\ar@{..>}@/_{1pc}/[dl]_{bc}\\
\Shv\left(S_{et}^{\prime}\right)\ar[rr]\sb(0.5){\gamma\circ\beta} & \vphantom{\,m} & \Shv_{G}^{\star}(S_{Zar}^{\prime})\ar@{^{(}->}[rr]\sp(0.5){\mathrm{inc}} &  & \Shv_{G}(S_{Zar}^{\prime})
}
\]
and we now have to show that the natural transformation of the first
square is an isomorphism $bc:f^{\ast}\circ(\gamma\circ\beta)\rightarrow(\gamma\circ\beta)\circ f^{\ast}$.
Since all functors in sight are left adjoints or equivalences, they
commute with all colimits. Since any etale sheaf $B$ on $S$ is a
colimit of representable sheaves, we may, by proposition~\ref{prop:IfLocHensEtcovbyFet},
restrict our attention to $B=\Hom_{S}(-,X)$ with $X$ in $S_{fet}$,
say with image $U$ in $S$. Then $f^{\ast}B=\Hom_{S'}(-,X')$ with
$X^{\prime}=X\times_{S}S'$ in $S_{fet}^{\prime}$ with image $U'=f^{-1}(U)$
in $S'$. Note that $X^{\prime}(\overline{\eta})\simeq X(\overline{\eta})$
as $G$-sets. By lemma~\ref{lem:DeltaOnFet}, we have
\[
\gamma\circ\beta(B)=j_{U!}(X(\overline{\eta})_{U})\quad\text{and}\quad\gamma\circ\beta(f^{\ast}B)=j_{U'!}(X^{\prime}(\overline{\eta})_{U'}).
\]
So we are reduced to showing that the following diagram is $2$-commutative:
\[
\xymatrix{ & \Shv_{G}(U_{Zar})\ar[dd]^{f^{\ast}}\ar[rr]\sp(0.5){j_{U!}} &  & \Shv_{G}(S_{Zar})\ar[dd]^{f^{\ast}}\\
\Set_{G}\ar[ur]\sp(0.45){(-)_{U}}\ar[dr]\sb(0.45){(-)_{U'}}\\
 & \Shv_{G}(U_{Zar}^{\prime})\ar[rr]\sp(0.5){j_{U'!}} &  & \Shv_{G}(S_{Zar}^{\prime})
}
\]
This follows from the obvious $2$-commutativity of the right adjoint
diagram:
\[
\xymatrix{ & \Shv_{G}(U_{Zar})\ar[dl]\sb(0.55){\Gamma(U,-)} &  & \Shv_{G}(S_{Zar})\ar[ll]\sb(0.5){j_{U}^{\ast}}\\
\Set_{G}\\
 & \Shv_{G}(U_{Zar}^{\prime})\ar[ul]\sp(0.55){\Gamma(U',-)}\ar[uu]_{f_{\ast}} &  & \Shv_{G}(S_{Zar}^{\prime})\ar[ll]\sb(0.5){j_{U'}^{\ast}}\ar[uu]_{f_{\ast}}
}
\]
This proves the proposition. 
\end{proof}
\begin{cor}
There is a right adjoint dual $2$-commutative diagram 
\[
\xyR{2pc}\xyC{3pc}\xymatrix{\AlgSp_{et}(S)\ar[r]^{\alpha} & \Shv(S_{et})\ar[r]\sp(0.45){\gamma\circ\beta} & \Shv_{G}^{\star}(S_{Zar})\\
\AlgSp_{et}(S')\ar[u]^{f_{\ast}^{et}}\ar[r]^{\alpha} & \Shv(S_{et}^{\prime})\ar[u]^{f_{\ast}}\ar[r]\sp(0.45){\gamma\circ\beta} & \Shv_{G}^{\star}(S_{Zar}^{\prime})\ar[u]^{f_{\ast}^{et}}
}
\]
\end{cor}

\begin{rem}
As in remark~\ref{rem:ZZZpushoutNOTequal}, while $f^{\ast}$ on
$\Shv_{G}^{\star}$ is the restriction of the eponymous functor on
$\Shv_{G}$, their respective right adjoints $f_{\ast}^{et}$ and
$f_{\ast}$ are related by
\[
\xymatrix{\Shv_{G}^{\star}(S_{Zar}) & \Shv_{G}(S_{Zar})\ar[l]\sb(0.45){(-)^{sm,\star}}\\
\Shv_{G}^{\star}(S_{Zar}^{\prime})\ar[u]^{f_{\ast}^{et}}\ar@{^{(}->}[r]\sp(0.5){\mathrm{inc}} & \Shv_{G}(S_{Zar}^{\prime})\ar[u]_{f_{\ast}}
}
\]
When $f$ is quasi-compact, $f_{\ast}$ takes smooth sheaves to smooth
sheaves, and we may thus replace the right column by $f_{\ast}:\Shv_{G}^{sm}(S_{Zar}^{\prime})\rightarrow\Shv_{G}^{sm}(S_{Zar})$. 
\end{rem}

\begin{cor}
For $A\in\AlgSp_{et}(S)$ with pull-back $A_{red}\in\AlgSp_{et}(S_{red})$,
\[
A\text{ is representable}\iff A_{red}\text{ is representable}.
\]
Moreover, $A_{red}(X\times_{S}S_{red})=A(X)$ for any proetale morphism
$X\rightarrow S$. 
\end{cor}

\begin{proof}
Apply the proposition to $f:S_{red}\rightarrow S$ and note that the
pull-back functor on Zariski sheaves $\Shv_{G}^{\star}(S_{Zar})\rightarrow\Shv_{G}^{\star}(S_{red,Zar})$
is the identity of $\Shv_{G}^{\star}(\left|S\right|)$. For the second
assertion, the proetale case follows from the etale case by \cite[\href{https://stacks.math.columbia.edu/tag/0468}{Tag 0468}]{SP},
which itself follows from the fact that $f^{\ast}$ is here an equivalence. 
\end{proof}

\subsection{Generic fiber}

Applying proposition~\ref{prop:BaseChange} to $\iota:\eta\rightarrow S$,
we obtain
\[
\xymatrix{\AlgSp_{et}(S)\ar[d]_{\iota^{\ast}}\ar[r]^{\alpha} & \Shv(S_{et})\ar[d]^{\iota^{\ast}}\ar[r]\sp(0.45){\gamma\circ\beta} & \Shv_{G}^{\star}(S_{Zar})\ar[d]^{\iota^{\ast}}\ar@(r,u)[dr]^{(-)_{\eta}}\\
\AlgSp_{et}(\eta)\ar[r]^{\alpha} & \Shv(\eta_{et})\ar[r]\sp(0.45){\gamma\circ\beta} & \Shv_{G}^{\star}(\eta_{Zar})\ar[r]\sp(0.55){\Gamma(\eta,-)} & \Set_{G}^{sm}
}
\]
where we have added the obviously $2$-commutative triangle at the
end, in which $\Gamma(\eta,-)$ is an equivalence. Passing to right
adjoints, we obtain 
\[
\xymatrix{\AlgSp_{et}(S)\ar[r]^{\alpha} & \Shv(S_{et})\ar[r]\sp(0.45){\gamma\circ\beta} & \Shv_{G}^{\star}(S_{Zar})\\
\AlgSp_{et}(\eta)\ar[u]^{\iota_{\ast}^{et}}\ar[r]^{\alpha} & \Shv(\eta_{et})\ar[u]^{\iota_{\ast}}\ar[r]\sp(0.45){\gamma\circ\beta} & \Shv_{G}^{\star}(\eta_{Zar})\ar[u]^{\iota_{\ast}^{et}}\ar[r]\sp(0.55){\Gamma(\eta,-)} & \Set_{G}^{sm}\ar@(u,r)[ul]^{(-)_{S}^{\star}}
}
\]
where the last adjunction is taken from (\ref{eq:Adj**S}), whose
counits are isomorphisms. It follows that our right adjoints are fully
faithful. The composite bottom equivalence is easily computed: it
takes $\mathcal{A}\in\AlgSp_{et}(\eta)$ to the $G$-set $\mathcal{A}(\overline{\eta})$.
We thus find that 
\[
\delta(\iota_{\ast}^{et}\mathcal{A})\simeq\mathcal{A}(\overline{\eta})_{S}^{\star}.
\]

\begin{prop}
The right adjoint $\iota_{\ast}^{et}$ factors through the strictly
full subcategory of algebraic spaces which are representable by separated
etale $S$-schemes.
\end{prop}

\begin{proof}
For $\mathcal{A}\in\AlgSp_{et}(\eta)$, set $A=\iota_{\ast}^{et}\mathcal{A}$
and $D=\mathcal{A}(\overline{\eta})_{S}^{\star}$, so that $D\simeq\delta(A)$,
and $D$ plainly belongs to $\Shv_{G}^{set}(S)$: we conclude by theorem~\ref{thm:Main}.
Concretely:
\[
D=\coprod_{c\in\pi_{0}(D)}D(c)\quad\text{with}\quad D(c)\simeq j_{U(c)!}(\gamma(c)_{U(c)})
\]
where $(U(c),\gamma(c))$ is the unique element of $\mathscr{S}_{G}^{m}(D)$
above $c\in G\backslash\mathcal{A}(\overline{\eta})$. By lemma~\ref{lem:EssImageOf*S},
$U(c)$ is the largest open $U$ of $S$ such that $c$ is fixed pointwise
by $I(U)$, namely 
\[
U(c)=\{s\in S:c\subset\mathcal{A}(\overline{\eta})^{I(s)}\}.
\]
Also, $\gamma(c)=c$ in $D(U(c))=\mathcal{A}(\overline{\eta})^{I(U(c))}$.
Let $X(c)$ be a connected finite etale $U(c)$-scheme with $X(c)(\overline{\eta})\simeq c$
as $\pi(U(c))$-sets. Then $D(c)\simeq\delta(X(c))$, so 
\[
A\simeq\Hom_{S}(-,X)\quad\text{with }X=\coprod_{c\in\pi_{0}(A)}X(c).
\]
Since all $X(c)$'s are separated over $S$, so is $X$.
\end{proof}

\subsection{The functor $S[-]^{et}$}

The previous proposition tells us that there is a functor 
\[
S[-]^{et}:\AlgSp_{et}(\eta)\rightarrow S_{set}
\]
whose composition with the Yoneda embedding 
\[
S_{set}\hookrightarrow S_{et}\hookrightarrow\AlgSp_{et}(S)
\]
is right adjoint to the generic fiber functor: for $A\in\AlgSp_{et}(S)$
and $\mathcal{A}\in\AlgSp_{et}(\eta)$, 
\[
\Hom_{\AlgSp(\eta)}(A_{\eta},\mathcal{A})\simeq\Hom_{\AlgSp(S)}(A,S[\mathcal{A}]^{et})
\]
where $A_{\eta}=\iota^{\ast}A$ is the generic fiber of $A$. The
counit $S[\mathcal{A}]_{\eta}^{et}\rightarrow\mathcal{A}$ is an isomorphism,
so $S[-]^{et}$ is fully faithful, and $\delta$ maps the unit $A\rightarrow S[A_{\eta}]^{et}$
to the unit $D\rightarrow(D_{\eta})_{S}^{\star}$, where $D=\delta(A)$.
For an irreducible $\mathcal{A}$, $S[\mathcal{A}]^{et}$ belongs
to the subcategory $\Fet_{U}^{c}$ of $S_{set}$, where $U$ is the
largest open of $S$ such that $I(U)$ acts trivially on $\mathcal{A}(\overline{\eta})$. 
\begin{prop}
If $S$ is normal, then $S[\mathcal{A}]^{et}$ is the largest open
$S[\mathcal{A}]^{\circ}$ of the normalization $S[\mathcal{A}]$ of
$S$ along $\mathcal{A}\rightarrow\eta\rightarrow S$ which is etale
over $S$.
\end{prop}

\begin{proof}
By~corollary~\ref{cor:SetisDisjointUnionOfFet}, we may write $S[\mathcal{A}]^{et}=\coprod X_{i}$
with irreducible $X_{i}$'s in $S_{fet}$. Then $\mathcal{A}=S[\mathcal{A}]_{\eta}^{et}=\coprod X_{i,\eta}$,
so $S[\mathcal{A}]=\coprod S[X_{i,\eta}]$ where $S[X_{i,\eta}]$
is the normalization of $S$ along $X_{i,\eta}\rightarrow\eta\rightarrow S$,
and $S[\mathcal{A}]^{\circ}=\coprod S[X_{i,\eta}]^{\circ}$ where
$S[X_{i,\eta}]^{\circ}$ is the largest open of $S[X_{i,\eta}]$ which
is etale over $S$. Let $U_{i}$ be the image of $X_{i}$ in $S$.
Since $X_{i}\rightarrow U_{i}$ is finite etale and $U_{i}$ is normal,
$X_{i}$ is normal and isomorphic to the normalisation $U_{i}[X_{i,\eta}]$
of $U_{i}$ along $X_{i,\eta}\rightarrow\eta\rightarrow U_{i}$. Since
$U_{i}[X_{i,\eta}]\simeq X_{i}$ is etale, the open $U_{i}[X_{i,\eta}]$
of $S[X_{i,\eta}]$ is contained in $S[X_{i,\eta}]^{\circ}$. We thus
obtain an open embedding 
\[
\xymatrix{S[\mathcal{A}]^{et}\simeq\coprod U_{i}[X_{i,\eta}]\ar@{^{(}->}[r] & \coprod S[X_{i,\eta}]^{\circ}=S[\mathcal{A}]^{\circ}}
\]
extending $\mathrm{Id}:\mathcal{A}\rightarrow\mathcal{A}$. Conversely
by the universal property of $S[\mathcal{A}]^{et}$, there is a unique
$S$-morphism $S[\mathcal{A}]^{\circ}\rightarrow S[\mathcal{A}]^{et}$
extending $\mathrm{Id}:\mathcal{A}\rightarrow\mathcal{A}$. Composing
it with our open embedding, we obtain an $S$-morphism $S[\mathcal{A}]^{\circ}\rightarrow S[\mathcal{A}]^{\circ}$
extending $\mathrm{Id}:\mathcal{A}\rightarrow\mathcal{A}$. Since
$S[\mathcal{A}]^{\circ}$ is separated etale over $S$, there is a
unique such morphism, namely the identity of $S[\mathcal{A}]^{\circ}$.
Our embedding is therefore surjective, and $S[\mathcal{A}]^{et}\simeq S[\mathcal{A}]^{\circ}$. 
\end{proof}

\subsection{The $(-)_{set}$ and $(-)_{et}$ functors}

For $A\in\AlgSp_{et}$, define
\[
A_{set}=\mathrm{Im}\left(A\rightarrow S[A_{\eta}]^{et}\right)\quad\text{and}\quad A_{et}=\mathrm{Im}\left(A\rightarrow S[A_{\eta}]^{et}\times A_{loc}\right).
\]
By the results of section~\ref{subsec:MonoEpiInAlgSpet}, 
\begin{itemize}
\item $A_{set}$ is open in $S[A_{\eta}]^{et}$ and thus belongs to $S_{set}$,
\item $A_{et}$ is open in $S[A_{\eta}]^{et}\times A_{loc}$ and thus belongs
to $S_{et}$,
\end{itemize}
and we have a diagram of surjective morphisms in $\AlgSp_{et}$
\[
\xyR{1ex}\xyC{2pc}\xymatrix{ &  & A_{set}\\
A\ar@{->>}[r] & A_{et}\ar@{->>}[ur]\ar@{->>}[dr]\\
 &  & A_{loc}
}
\]
which $\delta$ maps to the analogous diagram from section~\ref{subsec:Et/Set-sheaves},
\[
\xymatrix{ &  & D_{set}\\
D\ar@{->>}[r] & D_{et}\ar@{->>}[ur]\ar@{->>}[dr]\\
 &  & D_{G}
}
\]
This construction defines functors $A\mapsto A_{set}$ and $A\mapsto A_{et}$
which are left adjoint to the Yoneda embeddings $S_{set}\hookrightarrow\AlgSp_{et}$
and $S_{et}\hookrightarrow\AlgSp_{et}$, with units $A\twoheadrightarrow A_{set}$
and $A\twoheadrightarrow A_{et}$, the latter inducing an homeomorphism
on the underlying topological spaces. We have obtained $2$-commutative
diagrams of adjunctions
\[
\xyR{2pc}\xyC{3pc}\xymatrix{\AlgSp_{et}\ar[d]_{\delta}\ar@<2pt>[r]^{(-)_{et}} & S_{et}\ar[d]_{\delta}\ar@<2pt>[r]^{(-)_{set}}\ar@<2pt>@{^{(}->}[l]^{\mathrm{yon}} & S_{set}\ar[d]_{\delta}\ar@<2pt>[r]^{(-)_{\eta}}\ar@<2pt>@{^{(}->}[l]^{\mathrm{inc}} & \eta_{et}\ar[d]^{(-)(\overline{\eta})}\ar@<2pt>[l]^{S[-]^{et}}\\
\Shv_{G}^{\star}\ar@<2pt>[r]^{(-)_{et}} & \Shv_{G}^{et}\ar@<2pt>[r]^{(-)_{set}}\ar@<2pt>@{^{(}->}[l]^{\mathrm{inc}} & \Shv_{G}^{set}\ar@<2pt>[r]^{(-)_{\eta}}\ar@<2pt>@{^{(}->}[l]^{\mathrm{inc}} & \Set_{G}^{sm}\ar@<2pt>[l]^{(-)_{S}^{\star}}
}
\]
\[
\xyR{2pc}\xyC{3pc}\xymatrix{\AlgSp_{et}\ar[d]_{\delta}\ar@<2pt>[r]^{(-)_{et}} & S_{et}\ar[d]_{\delta}\ar@<2pt>[r]^{(-)_{loc}}\ar@<2pt>@{^{(}->}[l]^{\mathrm{yon}} & S_{loc}\ar[d]_{\delta}\ar@<2pt>@{^{(}->}[l]^{\mathrm{inc}}\ar[r]_{\sim}^{\left|-\right|} & \left|S\right|_{loc}\ar[d]^{\mathrm{sec}}\\
\Shv_{G}^{\star}\ar@<2pt>[r]^{(-)_{et}} & \Shv_{G}^{et}\ar@<2pt>[r]^{(-)_{G}}\ar@<2pt>@{^{(}->}[l]^{\mathrm{inc}} & \Shv\ar@<2pt>@{^{(}->}[l]^{\mathrm{inc}}\ar@{=}[r] & \Shv
}
\]

\subsection{Stalks\protect\label{subsec:Stalks}}

Since all morphisms in $\Fet_{S}$ are finite etale, we may define
\[
S\{\overline{\eta}\}=\underleftarrow{\lim}_{(X,x)\in\Fet_{S}(\overline{\eta})}X=\underleftarrow{\lim}_{(X,x)\in\Fet_{S}^{c}(\overline{\eta})}X=\underleftarrow{\lim}_{(X,x)\in\Fet_{S}^{g}(\overline{\eta})}X.
\]
Here $\Fet_{S}^{c}$ and $\Fet_{S}^{g}\subset\Fet_{S}^{c}$ are the
strictly full subcategories of $X$'s in $\Fet_{S}$ which are respectively
connected and Galois over $S$: the corresponding strictly full subcategories
$\Fet_{S}^{c}(\overline{\eta})$ and $\Fet_{S}^{g}(\overline{\eta})$
are initial in $\Fet_{S}(\overline{\eta})$. So $S\{\overline{\eta}\}$
is a connected proetale cover of $S$ with Galois group $\pi(S)=G/I(S)$.
By remark~\ref{rem:IrreducibilityUnibranchSImpliesEtalesLocIrred}
and \cite[8.2.9]{EGA4.3}, $S\{\overline{\eta}\}$ is irreducible;
by proposition~\ref{prop:IntegralOverLocHensIsHomeo}, $S\{\overline{\eta}\}\rightarrow S$
is an homeomorphism. Let $E(S)$ be the category of finite subextensions
$L$ of the Galois extension $k(\eta,S)=k(\eta,\overline{\eta})^{I(S)}$
of $k(\eta)$, where $k(\eta,\overline{\eta})$ is the separable closure
of $k(\eta)$ in $k(\overline{\eta})$. Mapping $(X,x)\in\Fet_{S}^{c}(\overline{\eta})$
to 
\[
L(X,x)=\text{image of }x^{\sharp}:\Gamma(X_{\eta},\mathcal{O}_{X_{\eta}})\rightarrow k(\overline{\eta})
\]
defines a $G$-equivariant equivalence of categories $\Fet_{S}^{c}(\overline{\eta})\rightarrow E(S)^{\mathrm{opp}}$.
An inverse functor takes $L\in E(S)$ to the finite etale $S$-scheme
$S[L]^{et}=S[\Spec(L)]^{et}$ equipped with the $\overline{\eta}$-valued
point given by $\overline{\eta}\rightarrow\Spec(L)\simeq S[L]_{\eta}^{et}\rightarrow S[L]^{et}$.
Thus also
\[
S\{\overline{\eta}\}=\underleftarrow{\lim}_{L\in E(S)}S[L]^{et}.
\]

For $s\in S$ and $U\ni s$ open in $S$, the base change functors
$\Fet_{S}\rightarrow\Fet_{U}\rightarrow\Fet_{S(s)}$ induce $G$-equivariant
proetale $S$-morphisms $S(s)\{\overline{\eta}\}\rightarrow U\{\overline{\eta}\}\rightarrow S\{\overline{\eta}\}$,
and 
\[
S(s)\{\overline{\eta}\}=\underleftarrow{\lim}_{s\in U\subset S}U\{\overline{\eta}\}=\underleftarrow{\lim}_{(X,x)\in S_{fet}(\overline{\eta},s)}X.
\]
Here $S(s)=\Spec(\mathcal{O}_{S,s})$ while $S_{fet}(\overline{\eta},s)=\cup_{s\in U}\Fet_{U}(\overline{\eta})$
is the category of all pairs $(X,x)$ where $X\in S_{fet}$ is such
that its image $U$ in $S$ contains $s$, and $x\in X(\overline{\eta})$.
For any such pair, we denote by $\tilde{x}:S(s)\{\overline{\eta}\}\rightarrow X$
the corresponding $S$-morphism. Since $S(s)\{\overline{\eta}\}\rightarrow S(s)$
is an homeomorphism, $S(s)\{\overline{\eta}\}$ is a local scheme.
Let $\tilde{s}\rightarrow S(s)\{\overline{\eta}\}$ be a geometric
point over the closed point $\overline{s}$ of $S(s)\{\overline{\eta}\}$.
This yields a geometric point of $S$ over $s$, and the corresponding
strict henselization of $\mathcal{O}_{S,s}$ is given by 
\[
\Spec(\mathcal{O}_{S,\tilde{s}}^{sh})=\underleftarrow{\lim}_{(Y,y)\in S_{et}(\tilde{s})}Y=\underleftarrow{\lim}_{(Y,y)\in S_{fet}(\tilde{s})}Y.
\]
Here $S_{et}(\tilde{s})$ is the category of pairs $(Y,y)$ with $Y\in S_{et}$
and $y\in Y(\tilde{s})$, and $S_{fet}(\tilde{s})$ is the strictly
full subcategory where $Y\in S_{fet}$, which is initial by proposition~\ref{prop:IfLocHensEtcovbyFet}.
Mapping $(X,x)$ to $(Y,y)$ where $Y=X\in S_{fet}$ and $y\in Y(\tilde{s})$
is the composition 
\[
\xymatrix{\tilde{s}\ar[r] & S(s)\{\overline{\eta}\}\ar[r]^{\tilde{x}} & X}
\]
defines an equivalence of categories $S_{fet}(\overline{\eta},s)\rightarrow S_{fet}(\tilde{s})$.
We thus obtain
\[
S(s)\{\overline{\eta}\}\simeq\Spec(\mathcal{O}_{S,\tilde{s}}^{sh}).
\]
A posteriori, we find that we can essentially take $\tilde{s}\rightarrow S(s)\{\overline{\eta}\}$
to be the closed immersion $\overline{s}\hookrightarrow S(s)\{\overline{\eta}\}$
of the closed point $\overline{s}$ of $S(s)\{\overline{\eta}\}$,
whose residue field $k(\overline{s})$ is indeed a separable closure
of $k(s)$. With these conventions: 
\begin{prop}
For any $s\in S$, there is a $2$-commutative diagram
\[
\xymatrix{\AlgSp_{et}(S)\ar[r]^{\alpha}\ar@(d,l)[dr]_{(-)(\overline{s})} & \Shv(S_{et})\ar[r]^{\gamma\circ\beta}\ar[d]^{(-)_{\overline{s}}} & \Shv_{G}^{\star}(S_{Zar})\ar@(d,r)[dl]^{(-)_{s}}\\
 & \Set_{G}^{sm}
}
\]
\end{prop}

\begin{proof}
Proposition~\ref{prop:CompStalks4Alpha} gives commutativity of the
first triangle, except that the target category was $\Set_{\Gamma(\overline{s})}$,
with $\Gamma(s)=\Gal(k(\overline{s})/k(s))$. Since $S(s)\{\overline{\eta}\}=\Spec(\mathcal{O}_{S,\overline{s}}^{sh})$,
\[
\Gamma(s)=\Aut(\mathcal{O}_{S,\overline{s}}^{sh}/\mathcal{O}_{S,s})=\Aut(S(s)\{\overline{\eta}\}/S(s))=\pi(S(s))=G/I(s).
\]
Since the third functor lands in the strictly full subcategory $\Set_{G}^{sm}$
of $\Set_{G}$, it remains to establish the $2$-commutativity of
the second triangle. Unwinding the definitions, we find that for $B\in\Shv(S_{et})$,
$C=\gamma_{2}\circ\gamma_{1}\circ\beta(B)$ and $D=\gamma_{3}(C)$,
\[
B_{\overline{s}}=\underrightarrow{\lim}_{(Y,y)\in S_{et}(\overline{s})}B(Y)=\underrightarrow{\lim}_{(Y,y)\in S_{fet}(\overline{s})}B(Y)
\]
\[
D_{s}=C_{s}=\underrightarrow{\lim}_{s\in U}C(U)=\underrightarrow{\lim}_{s\in U}\underrightarrow{\lim}_{(X,x)\in\Fet_{U}(\overline{\eta})}B(X)=\underrightarrow{\lim}_{(X,x)\in S_{fet}(\overline{\eta},s)}B(X)
\]
The equivalence $S_{fet}(\overline{\eta},s)\rightarrow S_{fet}(\overline{s})$
considered above between the indexing categories of these colimits
then yields the desired functorial isomorphism. 
\end{proof}
For a specialization $s^{\prime}\rightsquigarrow s$ in $S$, the
base change functor $\Fet_{S(s)}\rightarrow\Fet_{S(s')}$ induces
a $G$-equivariant $S$-morphism 
\[
\xymatrix{\Spec(\mathcal{O}_{S,\overline{s}'}^{sh})=S(s')\{\overline{\eta}\}\ar[r] & S(s)\{\overline{\eta}\}=\Spec(\mathcal{O}_{S,\overline{s}}^{sh}).}
\]
The localization morphism between Zariski stalk functors
\[
\loc:(-)_{s}\rightarrow(-)_{s'}
\]
corresponds to the localization morphism between etale stalk functors
\[
\loc:(-)_{\overline{s}}\rightarrow(-)_{\overline{s}'}
\]
induced by the functor $S_{et}(\overline{s})\rightarrow S_{et}(\overline{s}')$
mapping $(Y,y)$ to $(Y,y')$, with $y'$ given by
\[
\xyC{2pc}\xymatrix{\overline{s}^{\prime}\ar@{^{(}->}[r] & \Spec(\mathcal{O}_{S,\overline{s}'}^{sh})\ar[r] & \Spec(\mathcal{O}_{S,\overline{s}}^{sh})\ar[r]\sp(0.65){\tilde{y}} & Y}
\]
where $\tilde{y}$ is the canonical map $\Spec(\mathcal{O}_{S,\overline{s}}^{sh})=\underleftarrow{\lim}_{(Y,y)\in S_{et}(\overline{s})}Y\rightarrow Y$.
On $\AlgSp_{et}(S)$, it corresponds to the localization morphism
between geometric section functors 
\[
\loc:(-)(\overline{s})\rightarrow(-)(\overline{s}')
\]
whose evaluation at $A\in\AlgSp_{et}(S)$ is given by 
\[
\xymatrix{A(\overline{s}) & A\left(\mathcal{O}_{S,\overline{s}}^{sh}\right)\ar[l]\sb(0.56){\simeq}\ar[r] & A\left(\mathcal{O}_{S,\overline{s}'}^{sh}\right)\ar[r]\sp(0.56){\simeq} & A(\overline{s}')}
\]
The morphism $\Spec(\mathcal{O}_{S,\overline{s}'}^{sh})\rightarrow\Spec(\mathcal{O}_{S,\overline{s}}^{sh})$
is a proetale cover of its image, which is the inverse image of $S(s')$
in $\Spec(\mathcal{O}_{S,\overline{s}}^{sh})$. It follows that the
middle map factors as 
\[
\xymatrix{A\left(\mathcal{O}_{S,\overline{s}}^{sh}\right)\ar[r] & A\left(\mathcal{O}_{S,\overline{s}}^{sh}\otimes_{\mathcal{O}_{S,s}}\mathcal{O}_{S,s'}\right)\ar@{^{(}->}[r] & A\left(\mathcal{O}_{S,\overline{s'}}^{sh}\right).}
\]
In particular by proposition~\ref{prop:CaractOfsetAndet}, $A$ is
representable if and only if for all $s\in S$, $A\left(\mathcal{O}_{S,\overline{s}}^{sh}\right)\rightarrow A\left(\mathcal{O}_{S,\overline{s}}^{sh}\otimes_{\mathcal{O}_{S,s}}\mathcal{O}_{S,\eta}\right)$
is injective on $G$-orbits. 

\section{Henselian valuation rings}

We now apply the above results to the case where $S$ is the spectrum
of a valuation ring $\mathcal{O}$ with fraction field $K$, maximal
ideal $m$, residue field $k(m)=\mathcal{O}/m$, and value group $\Gamma=K^{\times}/\mathcal{O}^{\times}$,
whose group structure will be denoted additively. So if $v:K^{\times}\twoheadrightarrow\Gamma$
is the quotient map, $v(xy)=v(x)+v(y)$ and the formula 
\[
v(x)\geq v(y)\iff x\in\mathcal{O}y
\]
 turns $\Gamma$ into a totally ordered commutative group. We extend
$v$ to $K\rightarrow\Gamma\cup\{\infty\}$ by $v(0)=\infty$, so
that $\mathcal{O}=\{x\in K:v(x)\geq0\}$ and $m=\{x\in K:v(x)>0\}$. 

The set of $\mathcal{O}$-submodules of $K$ is totally ordered by
inclusion: if $I_{1}$ and $I_{2}$ are $\mathcal{O}$-submodules
of $K$ such that $I_{1}\not\subset I_{2}$, then for any $x_{1}\in I_{1}\setminus I_{2}$
and $x_{2}\in I_{2}$, $x_{1}\notin\mathcal{O}x_{2}$, so $v(x_{1})<v(x_{2})$,
hence $x_{2}\in mx_{1}\subset\mathcal{O}x_{1}\subset I_{1}$, whence
$I_{2}\subset I_{1}$. 

In particular, $S=\Spec(\mathcal{O})$ is totally ordered by inclusion.
This totally ordered set is not entirely random: it has a smallest
element $0$, a largest element $m$, and any subset $\mathcal{S}\neq\emptyset$
of $S$ has an inf and a sup in $S$, respectively given by 
\[
\mathrm{inf}(\mathcal{S})=\cap_{q\in\mathcal{S}}q\qquad\text{and}\qquad\sup(S)=\cup_{q\in\mathcal{S}}q.
\]
Thus all nonempty closed subsets of $S$ are irreducible, of the form
$V(p)=[p,m]$ for a unique $p\in S$, using standard notations for
intervals in posets. Accordingly, any open $U\neq S$ of $S$ is of
the form $[0,p[$ for a unique $p\in S$. For $p\subset q$ in $S$,
$p$ is a prime ideal of $\mathcal{O}_{q}$ and $\mathcal{O}(p,q)=\mathcal{O}_{q}/p$
is a valuation ring with spectrum $[p,q]$, fraction field $k(p)=\mathcal{O}_{p}/p$,
residue field $k(q)=\mathcal{O}_{q}/q$, and value group $\Gamma(p)/\Gamma(q)$,
where $\Gamma(p)=v(\mathcal{O}_{p}^{\times})$ and $\Gamma(q)=v(\mathcal{O}_{q}^{\times})$
are convex subgroups of $\Gamma$. If $\mathcal{O}$ is henselian,
then so are all $\mathcal{O}(p,q)$'s; in particular, $\mathcal{O}$
and all $\mathcal{O}(p,q)$'s are locally henselian. 
\begin{prop}
For an open $U\neq\emptyset$ of $S$, the following conditions are
equivalent:
\begin{enumerate}
\item $U$ is a local scheme.
\item $U$ is affine.
\item $U$ is quasi-compact.
\item $U$ is special, i.e.~$U=D(f)$ for some nonzero $f\in\mathcal{O}$.
\item $U=[0,p]$ for some $p\in\Spec(\mathcal{O})$.
\end{enumerate}
\end{prop}

\begin{proof}
Plainly $(1)\Rightarrow(2)\Rightarrow(3)$ and $(5)\Rightarrow(1)$
with $U=\Spec(\mathcal{O}_{p})$. For $(3)\Rightarrow(4)$: If $U$
is quasi-compact, it is covered by finitely many special opens $D(f_{i})=\Spec(\mathcal{O}_{f_{i}})$
for nonzero $f_{i}$'s in $\mathcal{O}$, and so $U=D(f)$ for any
$f\in\{f_{i}\}$ with $v(f)=\min\{v(f_{i})\}$. For $(4)\Rightarrow(5)$:
If $U=D(f)$, then $p=\cup_{q\in U}q$ belongs to $D(f)=U$, so $U=[0,p]$.
\end{proof}
\begin{defn}
A prime $p$ of $\mathcal{O}$ is \emph{special }if $[0,p]$ is open
in $\Spec(\mathcal{O})$.
\end{defn}

The map $p\mapsto[0,p]$ is an increasing bijection from special primes
of $\mathcal{O}$ to special opens of $\Spec(\mathcal{O})$. A prime
$p\neq m$ is special if and only if $\{r:p\subsetneq r\}$ has a
minimal element $q$; then $[0,p]=[0,q[=D(f)$ for any $f\in q\setminus p$
and the valuation ring $\mathcal{O}(p,q)$ has height $1$. The constructible
partitions of $\Spec(\mathcal{O})$ are given by 
\[
\Spec(\mathcal{O})=[0,p_{1}]\cup]p_{1},p_{2}]\cup\cdots\cup]p_{n-1},p_{n}]
\]
for finite sequences $p_{1}\subsetneq p_{2}\subsetneq\cdots\subsetneq p_{n}=m$
of special primes of $\mathcal{O}$.

Given the simple structure of $S$, a Zariski sheaf on $S$ is uniquely
characterized by its restriction to nonempty special opens, its sections
on such opens match the stalk at the corresponding special primes,
and the restriction maps between these spaces of sections match the
localization maps associated to the corresponding specializations
among special points. In other words, the category of Zariski sheaves
of sets $D$ on $S$ is equivalent to the category of functors $\mathscr{D}:\Sp(\mathcal{O})^{\circ}\rightarrow\Set$,
where $\Sp(\mathcal{O})$ is the totally ordered set of special points
in $\Spec(\mathcal{O})$, viewed as a category:
\[
\mathscr{D}(p)=D([0,p])=D_{p}\qquad\text{and}\qquad D(U)=\underleftarrow{\lim}_{p\in\Sp(\mathcal{O})\cap U}\mathscr{D}(p).
\]
In particular for all $q\in\Spec(\mathcal{O})$, 
\[
D([0,q[)=\underleftarrow{\lim}_{p\in\Sp(\mathcal{O}),p<q}\mathscr{D}(p)\quad\text{and}\quad D_{q}=\underrightarrow{\lim}_{p\in\Sp(\mathcal{O}),p\geq q}\mathscr{D}(p).
\]
Similar considerations apply to sheaves of $G$-sets. 

Suppose now that $\mathcal{O}$ is henselian and fix a geometric point
$\overline{\eta}\rightarrow S$ over the generic point $\eta$ of
$S$. Let $K^{sep}=k(\eta,\overline{\eta})$ be the separable closure
of $k(\eta)=K$ in $k(\overline{\eta})$ and set $G=\Gal(K^{sep}/K)$.
For $s\in U\subset S$, let $K(U)\subset K(s)\subset K^{sep}$ be
the fixed fields of $I(U)\supset I(s)$, and let $E(U)\subset E(s)$
be the finite extensions of $K$ in $K(U)$ and $K(s)$. So $E(\eta)$
is the set of all finite extensions of $K$ in $K^{sep}$, and 
\[
\begin{array}{rclrccl}
K(s) & = & \cup_{s\in U}K(U), &  & K(U) & = & \cap_{s\in U}K(s),\\
E(s) & = & \cup_{s\in U}E(U), &  & E(U) & = & \cap_{s\in U}E(s),
\end{array}
\]
by proposition~\ref{prop:LinkI(s)andI(U)}. For any integral $K$-algebra
$L$, let $S[L]$ be the normalization of $S$ in $\Spec(L)\hookrightarrow S$,
i.e.~$S[L]=\Spec(\mathcal{O}[L])$ where $\mathcal{O}[L]$ is the
integral closure of $\mathcal{O}$ in $L$. If $L$ is field, then
$S[L]\rightarrow S$ is an homeomorphism by proposition~\ref{prop:IntegralOverLocHensIsHomeo},
and for $s\in S$, we denote by $s_{L}\in S[L]$ the unique point
above $s$. For $L\in E(\eta)$, $S[L]\rightarrow S$ is etale at
$s_{L}$ if and only if $L\in E(s)$, and $S[L]^{et}=U[L]$ where
$U$ is the largest open of $S$ such that $L\in E(U)$. In particular
for a special $s\in\Sp(\mathcal{O})$, 
\[
I(s)=I([0,s]),\qquad K(s)=K([0,s])\quad\text{and}\quad E(s)=E([0,s]).
\]
With notations as in section~\ref{subsec:Stalks}, we have 
\[
S\{\overline{\eta}\}=\underleftarrow{\lim}_{L\in E(S)}S[L]=S[K(S)]
\]
\[
S(s)\{\overline{\eta}\}=\underleftarrow{\lim}_{s\in U}U\{\overline{\eta}\}=\underleftarrow{\lim}_{L\in E(s)}S(s)[L]=S(s)[K(s)]=\Spec(\mathcal{O}_{S,\overline{s}}^{sh})
\]
where $\overline{s}=s_{K(s)}$ is the closed point of $S(s)\{\overline{\eta}\}$. 

\subsection*{Summary}

The category of etale algebraic spaces $A$ over $S$ and the category
of etale sheaves $B$ on $S$ are equivalent to the category of presheaves
of smooth $G$-sets $\mathscr{D}$ on $\Sp(\mathcal{O})$ such that
for all special prime $s$ of $\mathcal{O}$, $I(s)$ acts trivially
on $\mathscr{D}(s)$, with
\[
\mathscr{D}(s)=A(\overline{s})=A(\mathcal{O}_{S,\overline{s}}^{sh})=B_{\overline{s}}=\underrightarrow{\lim}_{L\in E(s)}B(S[L]^{et})=\underrightarrow{\lim}_{L\in E(s)}B(S(s)[L]).
\]
The representable (resp. representable and separated) objects correspond
to those $\mathscr{D}'s$ such that for every $s'\subset s$ in $\Sp(\mathcal{O})$,
the localization map $\mathscr{D}(s)\rightarrow\mathscr{D}(s')$ is
injective on $G$-orbits (resp.~injective). Under these equivalences, 
\begin{itemize}
\item $A\mapsto A_{\eta}$ corresponds to $\mathscr{D}\mapsto\mathscr{D}_{\eta}=\underrightarrow{\lim}_{s\in\Sp(\mathcal{O})}\mathscr{D}(s)$,
\item $A\mapsto A_{set}$ to $\mathscr{D}\mapsto\mathscr{D}_{set},$ with
$\mathscr{D}_{set}(s)=\mathrm{Im}(\mathscr{D}(s)\rightarrow\mathscr{D}_{\eta})$,
\item $A\mapsto A_{loc}$ to $\mathscr{D}\mapsto\mathscr{D}_{loc},$ with
$\mathscr{D}_{loc}(s)=G\backslash\mathscr{D}(s)$,
\item $A\mapsto A_{et}$ to $\mathscr{D}\mapsto\mathscr{D}_{et},$ with
$\mathscr{D}_{et}(s)=\mathrm{Im}(\mathscr{D}(s)\rightarrow\mathscr{D}_{\eta}\times G\backslash\mathscr{D}(s))$.
\end{itemize}
\bibliographystyle{plain}
\bibliography{MyBib}

\end{document}